\documentclass[12pt,fleqn]{article}

\usepackage{amsmath,amssymb,amsthm}
\usepackage{geometry} 
\usepackage[pdfpagemode=UseNone,pdfstartview=FitH]{hyperref}
\usepackage{color,tikz-cd,float}

\newcommand{\bgset}[1]{\big\{#1\big\}}

\newcommand{\F}{{\mathcal F}}

\newcommand{\N}{\mathcal N}

\newenvironment{enumroman}{\begin{enumerate}

}{\end{enumerate}}

\newtheorem{lemma}{Lemma}[section]
\newtheorem{proposition}[lemma]{Proposition}
\newtheorem{theorem}[lemma]{Theorem}
\newtheorem{conjecture}[lemma]{Conjecture}

\theoremstyle{definition}

\theoremstyle{remark}
\newtheorem{remark}[lemma]{Remark}

\numberwithin{equation}{section}

\title{Prescribed energy solutions of concave-convex type problems involving sign-changing or vanishing weights \bf}

\author{\bf Kanishka Perera\\
Department of Mathematics\\
Florida Institute of Technology\\
150 W University Blvd, Melbourne, FL 32901-6975, USA\\
\em kperera@fit.edu\\
[\medskipamount]
\bf Humberto Ramos Quoirin\\
	CIEM-Conicet\\
    Universidad Nacional de C\'{o}rdoba\\
(5000) C\'{o}rdoba, Argentina\\
	\em humbertorq@gmail.com\\
[\medskipamount]
\bf Kaye Silva\\
Instituto de Matem\'{a}tica e Estat\'{i}stica\\
Universidade Federal de Goi\'{a}s\\
Rua Samambaia, 74001-970 Goi\^{a}nia, GO, Brazil\\
\em kayesilva@ufg.br}
\date{}

\begin{document}

\maketitle

\begin{abstract} 
We provide an abstract approach to find couples $(\lambda,u) \in \mathbb{R} \times X$ satisfying $$\Phi_\lambda(u)=c \quad \mbox{and} \quad \Phi'_\lambda(u)=0,$$
for some suitable values of $c \in \mathbb{R}$. Here $\Phi_\lambda$ is a $C^1$ functional (set on a Banach space $X$) whose main prototype is the energy functional associated to a concave-convex problem with sign-changing or vanishing weights. This approach allows us to derive several existence, multiplicity and bifurcation type results for the equation $\Phi'_\lambda(u)=0$ with $\lambda$ fixed.
\end{abstract}

\begin{center}
	\begin{minipage}{12cm}
		\tableofcontents
	\end{minipage}
\end{center}

\newpage

\section{Introduction}

\subsection{Concave-convex problems with sign-changing or vanishing weights} 

In this work we deal with nonlinear elliptic problems having a concave-convex type structure. The model equation for us is the boundary value problem

\begin{equation}
	\label{quasi1}
	\begin{cases}
		-\Delta_p u=\lambda a(x)|u|^{\alpha-2}u +b(x)|u|^{\beta-2}u &\mbox{ in } \Omega, \\
		u=0 &\mbox{ on } \partial \Omega,
	\end{cases}
\end{equation}
where $1<\alpha <p<\beta<p^*$, $\Omega$ is a bounded domain of $\mathbb{R}^N$ and $\lambda\in \mathbb{R}$ is a parameter. We assume, for simplicity, that $a, b \in L^{\infty}(\Omega)$. In the particular case where $a \equiv b \equiv 1$ and $p = 2$, problem \eqref{quasi1} reduces to the classical Ambrosetti--Brezis--Cerami problem~\cite{AmBrCe}. Here we are mainly interested in the case where $a$ and $b$ can vanish or change sign.

A pioneering contribution addressing equation~\eqref{quasi1} with sign-changing weights $a$ and $b$ is due to~\cite{FiGoUb}, where the existence of two nontrivial nonnegative solutions was established for small values of $\lambda>0$. To this end, the authors combine the Mountain Pass Theorem, local minimization techniques, and the method of lower and upper solutions. Let us mention that the results of \cite{FiGoUb} hold for a larger class of nonlinearities, and include non-existence results as well.

Subsequently, the methods in~\cite{FiGoUb} were simplified in~\cite{BrWu} (see also~\cite{BrWu1}) for the powerlike case \eqref{quasi1}. In that setting, the existence of two positive solutions for small $\lambda > 0$ was obtained by minimizing the energy functional over two disjoint connected components of the Nehari manifold.

Regarding the existence of infinitely many solutions of \eqref{quasi1} with $a \equiv b \equiv 1$, we refer to \cite{AmBrCe,BW} for the case $p=2$ and to \cite{GP} for the case $\beta=p^*$. 
To the best of our knowledge, the only result of this nature for equation~\eqref{quasi1} with $a$ and $b$ changing sign was proved in~\cite{Gu}. There, under the assumption that the set $\{x \in \Omega : a(x) > 0\}$ has nonempty interior, a sequence of solutions with negative energy was constructed for all $\lambda > 0$. Let us also mention \cite{AGP,BM} for some bifurcation results when $a \equiv b \equiv 1$ and $\Omega$ is either a ball or an annulus.

For related equations with similar concave--convex nonlinearities, we highlight ~\cite{Wu}, which deals with the operator $-\Delta u + u$ in $\mathbb{R}^N$, and~\cite{GoSr}, which studies an equation involving the fractional $p$-Laplacian operator on a bounded domain. In both cases, the existence of finitely many solutions was proven. We also refer the reader to~\cite{AmBoPe,Ha,Il,LeQuSi,KaQuUm,Pa,SiMa}, where several results are established for problems similar to \eqref{quasi1}, typically under the assumption that either $a$ or $b$ does not change sign.

Our proposal in this work is to look for couples $(\lambda,u)$ that solve \eqref{quasi1} and, in addition, satisfy $\Phi_\lambda(u)=c$, for a given $c \in \mathbb{R}$. Here $\Phi_\lambda$ is the energy functional associated to \eqref{quasi1}, i.e.
\begin{equation}\label{dephi}
    \Phi_\lambda(u)=\frac{1}{p}\int_\Omega |\nabla u|^p dx - \frac{\lambda}{\alpha}\int_\Omega a(x)|u|^{\alpha} dx- \frac{1}{\beta}\int_\Omega b(x)|u|^{\beta} dx, \quad u \in W_0^{1,p}(\Omega).  
\end{equation}

We shall follow an abstract approach based on Nehari subsets and their topological properties, which can be used to study many other problems in addition to \eqref{quasi1}. As a consequence we shall prove the existence of infinitely many solutions (in some case two sequences of solutions) of \eqref{quasi1} for some fixed $\lambda$. We aim at extending the results of \cite[Section 5]{MR4736027} (see also \cite{LeQuSi}), which apply in particular to \eqref{dephi} with $a,b>0$ in $\Omega$. We shall now treat functionals containing homogeneous terms that can vanish and change sign.

\subsection{Main abstract result - the prescribed energy problem}

Given a uniformly convex Banach space $X$, equipped with $\|\cdot\|\in C^1(X\setminus\{0\})$, we consider the functional
\begin{equation*}
	\Phi_\lambda=I_1-\lambda I_2,
\end{equation*}
where $\lambda\in \mathbb{R}$ is a given parameter, and $I_1,I_2\in C^1(X)$ are even functionals with $I_1(0)=I_2(0)=0$. The {\it prescribed energy problem} for this family of functionals consists in finding critical points of $\Phi_\lambda$ at a prescribed critical level. More precisely, given $c \in \mathbb{R}$ we look for couples $(\lambda,u)\in \mathbb{R}\times (X\setminus\{0\})$ such that 
\begin{equation}\label{PEP}
\tag{PEP}\Phi_\lambda(u)=c \quad \mbox{and} \quad \Phi'_\lambda(u)=0.
\end{equation} 
Let $u\in X$ be such that $I_2(u)\neq 0$. We see that
\begin{equation*}
\Phi_\lambda(u)=c\quad \mbox{if, and only if},\quad \lambda=\lambda(c,u):=\frac{I_1(u)-c}{I_2(u)}.
\end{equation*}
Moreover
\begin{equation}\label{derivatives}
	\frac{\partial \lambda}{\partial u}(c,u)=\frac{\Phi'_{\lambda(c,u)}(u)}{I_2(u)}.
\end{equation}
Therefore we conclude that
if 	$\lambda=\lambda(c,u)$ and 	$ \frac{\partial \lambda}{\partial u}(c,u)=0$ then $\Phi_\lambda(u)=c$ and   $\Phi_\lambda'(u)=0$.

Given $u\in X$ such that $I_2(u)\neq 0$ and $c\in \mathbb{R}$, let us set
\begin{equation*}
	\varphi_{c,u}(t):=\lambda(c,tu), \quad \forall t>0.
\end{equation*}
We assume the following condition:

\begin{enumerate}
	\item[(H1)] There exists an open set $I\subset \mathbb{R}$ and an open cone $\mathcal{C} \subset \{u \in X: I_2(u)\neq 0\}$ such that:
	\begin{enumerate}
		\item[(a)] the map $(c,u,t)\mapsto \varphi_{c,u}'(t)$ belongs to $C^1(I\times \mathcal{C}\times(0,\infty))$;
		\item[(b)] for every $(c,u)\in I\times \mathcal{C}$ the map $\varphi_{c,u}$ has exactly one local minimizer $t^+(c,u)>0$ of Morse type or for every $(c,u)\in I\times \mathcal{C}$ the map $\varphi_{c,u}$ has exactly one local maximizer $t^-(c,u)>0$ of Morse type. 
	\end{enumerate}
\end{enumerate}

We denote
\begin{equation*}
\mathcal{N}_c^\pm:=\{t^{\pm}(c,u)u:\ u\in \mathcal{C}\},
\end{equation*}
and note, by condition $(H1)$, that $\mathcal{N}_c^\pm$ is a $C^1$-Finsler manifold  contained in the Nehari set
\begin{equation*}
	\mathcal{N}_c:=\left\{u\in\mathcal{C}:\ 	\frac{\partial \lambda}{\partial u}(c,u)u=0\right\}.
\end{equation*}
Let us write, for simplicity, $t(c,u)=t^\pm(c,u)$, and set
\begin{equation*}
	\Lambda(c,u):=\varphi_{c,u}(t(c,u))=\lambda(c,t(c,u)u), \quad \forall u \in \mathcal{C}.
\end{equation*}
Condition $(H1)$ implies that for each $c\in I$, the functional $u \mapsto \Lambda(c,u)$ belongs to $C^1(\mathcal{C})$ and
\begin{equation}\label{equi}
	\frac{\partial \Lambda}{\partial u}(c,u)=0\quad \mbox{if, and only if}\quad 	\frac{\partial \lambda}{\partial u}(c,t(c,u)u)=0.
\end{equation}
Observe that $\Lambda(c,u)$ is the restriction of $\lambda(c,\cdot)$ to $\mathcal{N}_c^\pm$. Indeed, it is clear by $(H1)$ that $\Lambda(c,u)=\varphi_{c,u}(1)=\lambda(c,u)$ if $u\in \mathcal{N}_c^\pm$. Let us denote by
\begin{equation*}
	\mathcal{S}=\{u\in X: \|u\|=1\}
\end{equation*}
the unit sphere of $X$ and set
\begin{equation*}
 	\mathcal{S}_\mathcal{C}=\mathcal{S}\cap \mathcal{C}.
 \end{equation*}
Then $\mathcal{S}_\mathcal{C}$ is a $C^1$-Finsler manifold, symmetric and, by $(H1)$, it is diffeomorphic to $\mathcal{N}_c^\pm$ through the map $u \mapsto t(c,u)u$, $u\in \mathcal{S}_\mathcal{C}$. The proof of these facts is just an application of the Implicit Function Theorem (see \cite[Section 3]{MR4736027} where $X$ has to be replaced by $\mathcal{C}$).

A crucial difference between our situation and the one considered in \cite{MR4736027} is the fact that now $\mathcal{N}_c^\pm$ (and possibly $	\mathcal{S}_\mathcal{C}$) does not need to be a complete manifold with respect to the Finsler Metric.

From the previous discussion it follows that we can find couples $(\lambda,u)\in \mathbb{R}\times (X\setminus\{0\})$ solving $\Phi_\lambda(u)=c$ and $\Phi'_\lambda(u)=0$ by looking for critical points of the map $u \mapsto \Lambda(c,u)$. Since this map is $0$-homogeneous, we shall deal with its restriction to $\mathcal{S}_\mathcal{C}$, i.e. the map
\begin{equation*}
	\widetilde{\Lambda}(c,\cdot)=\Lambda_{|\mathcal{S}_\mathcal{C}}(c,\cdot). 
\end{equation*}

Let $\F$ denote the class of closed and symmetric subsets of $\mathcal{S}_\mathcal{C}$. Given $M\in \F$ let $\gamma(M)$ denote its Krasnoselskii genus. For $k \ge 1$ denote $$\F_{k} = \bgset{M \in \F : M\ \mbox{is compact and}\  \gamma(M) \ge k}.$$
We set
\begin{equation*}
	\gamma(\mathcal{S}_\mathcal{C}):=\sup\{k\in \mathbb{N}:\ \F_k\neq \emptyset\}
\end{equation*}
and
\begin{equation} \label{50}
	\lambda_{c,k}:= \inf_{M \in \F_{k}}\, \sup_{u \in M}\, \widetilde{\Lambda}(c,u),\ \mbox{for } c \in I,\mbox{ and } k\le \gamma(\mathcal{S}_\mathcal{C}).
\end{equation}
The following assumptions will be used to show that $\lambda_{c,k}$ is a critical level to $\widetilde{\Lambda}(c,\cdot)$:

\begin{enumerate}
	\item[(H2)] 
	
	\begin{enumerate}
		\item[(a)] for any $c\in I$, the functional $u \mapsto \widetilde{\Lambda}(c,u)$ is bounded from below in $\mathcal{S}_\mathcal{C}$;
		\item[(b)] for any $c\in I$ and $k\le \gamma(\mathcal{S}_\mathcal{C})$ the functional $u \mapsto \widetilde{\Lambda}(c,u)$ satisfies the Palais--Smale condition at the level $\lambda_{c,k}$;
		\item[(c)] if $(u_n)\subset \mathcal{S}_\mathcal{C}$ satisfies $I_2(u_n)\to 0$, then $\widetilde{\Lambda}(c,u_n)\to \infty$.
	\end{enumerate}
\end{enumerate}
We note that $(H2)$-$(c)$ provides us a control of $\widetilde{\Lambda}(c,\cdot)$ near the ``boundary" of $S_{\mathcal{C}}$.

\begin{theorem}\label{thm1} Suppose $(H1)$ and $(H2)$, and let $\lambda_{c,k}$ be given by \eqref{50}. Then for any $c \in I$ and $1 \le k\le \gamma(\mathcal{S}_\mathcal{C})$ there exists $u_{c,k}\in \mathcal{C}$ such that 
	\begin{equation*}
		\Phi_{\lambda_{c,k}}(\pm u_{c,k})=c\ \ \mbox{and}\ \ 	\Phi_{\lambda_{c,k}}'(\pm u_{c,k})=0.
	\end{equation*}
Moreover, if $\gamma(\mathcal{S}_\mathcal{C})=\infty$ and $\widetilde{\Lambda}(c,\cdot)$ satisfies the Palais--Smale condition at any level, then $(\lambda_{c,k})$ is a nondecreasing unbounded sequence.
\end{theorem}

Let us note that our method does not provide us with solutions couples $(\lambda,u)$ of \eqref{PEP} satisfying $I_2(u)=0$. For such solutions the problem reduces to
\begin{equation}\label{pred}
I_1(u)=c \quad \mbox{and} \quad I_1'(u)=0,
\end{equation} 
and any $\lambda \in \mathbb{R}$ yields a solution couple of \eqref{PEP}. Note also that \eqref{pred}
 has a nontrivial solution only if $c>0$. In the model case \eqref{quasi1} with $a \ge 0$ such that $\Omega_0:=a^{-1}(0)$ is a smooth nonempty domain, the condition $I_2(u)=0$ corresponds to $\int_\Omega a(x)|u|^{\alpha}=0$ i.e. $u \in W_0^{1,p}(\Omega_0)$, so that the problem \eqref{pred} becomes
$$-\Delta_p u=b(x)|u|^{\beta-2}u, \quad u \in W_0^{1,p}(\Omega_0), \quad \int_{\Omega_0} |\nabla u|^p=\frac{\beta p c}{\beta-p}.$$

\subsection{Abstract concave-convex problems with sign-changing or vanishing weights}

As an application of Theorem \ref{thm1} we consider an abstract functional inspired by concave-convex problems with sign-changing weights (see for example \cite{BrWu1}). Let us deal with the class of functionals
\begin{equation}\label{dephi}
	\Phi_\lambda(u)=\frac{1}{\eta}N(u)-\frac{\lambda}{\alpha}A(u)-\frac{1}{\beta}B(u),\quad u\in X,
\end{equation}
where $1<\alpha<\eta<\beta$, and $N,A,B\in C^1(X)$ are even functionals satisfying the following additional conditions:

\begin{enumerate}
	\item[(C1)] $N,A,B$ are $\eta$-homogeneous, $\alpha$-homogeneous and $\beta$-homogeneous, respectively.
	\item[(C2)] There exists $C,C'>0$ such that $C'\|u\|^\eta \ge N(u)\ge C^{-1}\|u\|^\eta$, $|A(u)|\le C\|u\|^\alpha$ and $|B(u)|\le C\|u\|^ \beta$ for all $u\in X$.
	\item[(C3)] If $(\lambda_n) \subset \mathbb{R}$ and $(u_n) \subset X$ are bounded sequences such that $(\Phi_{\lambda_n}(u_n))$ is bounded and $\Phi_{\lambda_n}'(u_n)\to 0$, then $(u_n)$ has a convergent subsequence. 
\end{enumerate}
We shall deal with the sets $$\mathcal{C}_A:=\{u\in X:\ A(u)>0\} \quad \mbox{and} \quad \mathcal{C}_B:=\{u\in X:\ B(u)>0\},$$
which are open cones of $X$, in view of the continuous and homogeneous behavior of $A$ and $B$.
Let us note that the definition of $\gamma(\mathcal{C}_A)$ is similar to the one of $\gamma(\mathcal{S}_\mathcal{C})$. Moreover, by $(C1)$ it is clear that $\gamma(\mathcal{C}_A)=\gamma(S_{\mathcal{C}_A})$.
We are mainly interested in the case where $\gamma(\mathcal{C}_A)=\gamma(\mathcal{C}_B)=\gamma(\mathcal{C}_A \cap \mathcal{C}_B)=\infty$, which happens in our applications. However, our abstract results only require $\gamma(\mathcal{C}_A)>1$ or $\gamma(\mathcal{C}_A \cap \mathcal{C}_B)>1$.

\begin{remark} We point out that condition $(C1)$, together with the continuity of $N$, $A$, and $B$, implies the inequalities from above in condition $(C2)$ (see \cite[Proposition 1.1]{PeAgOr}). However, for the sake of clarity and simplicity, we will state $(C2)$ in this form.
	
\end{remark}

\begin{theorem}\label{thm2} Suppose $(C1)$-$(C3)$. Then there exist $c^*<0 <c^{**}$ such that: 
	\begin{enumerate}
		\item[i)] For any $1\le k\le \gamma(\mathcal{C}_A)$ and $c\in (c^*,0)$ there exist $\lambda_{c,k}^+>0$ and $v_{c,k}\in \mathcal{C}_A$ such that 
			\begin{equation*}
			\Phi_{\lambda_{c,k}^+}(v_{c,k})=c\ \ \mbox{and}\ \ 	\Phi_{\lambda_{c,k}^+}'(v_{c,k})=0.
		\end{equation*}
\item[ii)] For any $1\le k\le \gamma(\mathcal{C}_A)$ the map $c \mapsto \lambda_{c,k}^+$  is continuous and decreasing in $ (c^*,0)$, and satisfies $\displaystyle \lim_{c\to 0^-}\lambda_{c,k}^+=0$.
\item[iii)] If $\gamma(\mathcal{C}_A)=\infty$ then $(\lambda_{c,k}^+)$ is a nondecreasing unbounded sequence, i.e. $0< \lambda_{c,k}^+ \le \lambda_{c,k+1}^+ \to \infty$ as $k\to \infty$, for any $c\in (c^*,0)$.
\item[iv)] If $\gamma(\mathcal{C}_A)=\infty$ then for each $\lambda>0$ there exist sequences $(v_n)\subset \mathcal{C}_A$, $(c_n)\subset (c^*,0)$ and $(k_n)\subset \mathbb{N}$ such that $c_n\to 0$, $k_n\to \infty$, and
\begin{equation*}
		\lambda=\lambda_{c_n,k_n}^+, \ \		\Phi_{\lambda_{c_n,k_n}^+}(v_n)=c_n\ \ \mbox{and}\ \ 	\Phi_{\lambda_{c_n,k_n}^+}'(v_n)=0, \quad \mbox{for every } n.
			\end{equation*}
Moreover $v_n\to 0$ in $X$, so $(\lambda,0)$ is a bifurcation point for any $\lambda>0$.
\item[v)] For any $k\le \gamma(\mathcal{C}_A\cap \mathcal{C}_B)$ and $c\in (c^*,c^{**})$ there exist $\lambda_{c,k}^->0$ and $u_{c,k}\in \mathcal{C}_A\cap\mathcal{C}_B$ such that
\begin{equation*}
			\Phi_{\lambda_{c,k}^-}(u_{c,k})=c\ \ \mbox{and}\ \ 	\Phi_{\lambda_{c,k}^-}'(u_{c,k})=0.
		\end{equation*}
Moreover $\lambda_{c,k}^+<\lambda_{c,k}^-$ for every $c\in (c^*,0)$.
\item[vi)] For any $1\le k\le \gamma(\mathcal{C}_A\cap \mathcal{C}_B)$ the map $c \mapsto \lambda_{c,k}^-$ is continuous and decreasing in $(c^*,c^{**})$.
\item[vii)] If $\gamma(\mathcal{C}_A\cap \mathcal{C}_B)=\infty$ then  $(\lambda_{c,k}^-)$ is a nondecreasing unbounded sequence, , i.e. $0< \lambda_{c,k}^- \le \lambda_{c,k+1}^- \to \infty$ as $k\to \infty$, for any  $c\in (c^*,c^{**})$.
\end{enumerate}

\end{theorem}

\begin{figure}[H]
	\centering
	\begin{tikzpicture}[>=latex,scale=1.2]
		\draw[->] (-1,0) -- (6,0) node[below] {\scalebox{0.8}{$\lambda$}};
		\foreach \x in {}
		\draw[shift={(\x,0)}] (0pt,2pt) -- (0pt,-2pt) node[below] {\footnotesize $\x$};
		\draw[->] (0,-2) -- (0,3) node[left] {\scalebox{0.8}{$c$}};
		\foreach \y in {}
		\draw[shift={(0,\y)}] (2pt,0pt) -- (-2pt,0pt) node[left] {\footnotesize $\y$};
		\draw[blue,thick] (1,1) .. controls (3,-1.4) .. (3,-1.4);
		\draw[blue,thick] (2,1) .. controls (4,-1.4) .. (4,-1.4);
		
		\draw[blue,thick] (3,1) .. controls (5,-1.4) .. (5,-1.4);
		\draw[red,thick] (0,0) .. controls (1,-.1) .. (2.5,-1.4);
		\draw[red,thick] (0,0) .. controls (1.5,-.1) .. (3.3,-1.4);
		\draw[red,thick] (0,0) .. controls (2,-0.1) .. (4.3,-1.4);
		\draw [thick] (1.8,-2) -- (1.8,3);
		\draw (1.8,-.18) node{\scalebox{0.8}{$\bullet$}};
		\draw (1.8,-.37) node{\scalebox{0.8}{$\bullet$}};
		\draw (1.8,-.82) node{\scalebox{0.8}{$\bullet$}};
				\draw (1,1) node[above]{\scalebox{0.8}{$\lambda_{c,1}^-$}} ; 	
				\draw  (2,1) node[above]{\scalebox{0.8}{$\lambda_{c,2}^-$}} ; 
		\draw  (3,1) node[above]{\scalebox{0.8}{$\lambda_{c,k}^-$}} ; 
		\draw  (1,-.5) node[below]{\scalebox{1.5}{$\nexists$}}  ; 	
				\draw  (2.5,-2) node[above]{\scalebox{0.8}{$\lambda_{c,1}^+$}} ; 	
				\draw (3.3,-2) node[above]{\scalebox{0.8}{$\lambda_{c,2}^+$}} ; 
				\draw  (4.3,-2) node[above]{\scalebox{0.8}{$\lambda_{c,k}^+$}} ; 
		\draw [thick] (-.1,-1.4) node[left]{\scalebox{0.8}{$c^*$}} -- (.1,-1.4); 
		\draw [thick] (-.1,1) node[left]{\scalebox{0.8}{$c^{**}$}} -- (.1,1); 
		\draw [thick,dashed] (0,-1.4) -- (6,-1.4);
		\draw [thick,dashed] (0,1) -- (6,1);
		\draw  (2.3,2.4) node[above]{\scalebox{0.8}{$\lambda=\overline{\lambda}$}} ;
	\end{tikzpicture}
\caption{Energy curves for Theorem \ref{thm2}. The red curves correspond to $(\lambda_{c,k}^+,c)$, $c\in (c^*,0)$ and the blue ones to $(\lambda_{c,k}^-,c)$, $c\in (c^*,c^{**})$. }\label{fig:CC}
\end{figure}
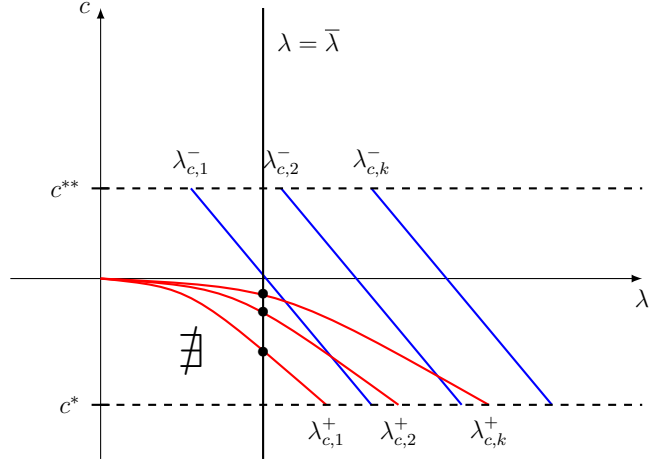

\begin{remark}
The value $\lambda_{c,1}^+$ in Theorem \ref{thm2} is, for any $c\in (c^*,0)$, the ground state level of the functional $u \mapsto \lambda(c,u)$ over $\mathcal{C}_A$. Indeed, we can write
$$\lambda_{c,1}^+=\inf_{u \in \mathcal{S}_{\mathcal{C}_A}} \widetilde{\Lambda}(c,u)=\inf_{u \in \mathcal{N}_c \cap {\mathcal{C}_A}} \lambda(c,u).$$
In addition, one may check that whenever $c<0$ any solution $(\lambda,u)$ of the {\it prescribed energy  problem} \eqref{PEP}  satisfies $\lambda A(u)>0$. It follows that for $c\in (c^*,0)$ the  problem \eqref{PEP}  has no solution with $0<\lambda<\lambda_{c,1}^+$.
\end{remark}

It is interesting to note here that our results concerning $\lambda_{c,k}^-$ are quite different from \cite[Theorem 5.14]{MR4736027}. The difference, as we shall see later, comes from the fact that condition $(H2)$ is not clear when $\widetilde{\Lambda}^-(c,\cdot)$ assumes negative levels. We note that the value $c^{**}$ appears to ensure that  $\widetilde{\Lambda}^-(c,\cdot)>0$ for $c<c^{**}$. In Sections \ref{abstractarguments} and \ref{Sfurther} we shall discuss more about this issue (see also Conjecture \ref{conje}).

 Theorem \ref{thm2} has the following counterpart, if we assume that the set 
 $$\mathcal{C}_{-A}:=\{u\in X:\ A(u)<0\}$$
  is large enough (i.e. $\gamma(\mathcal{C}_{-A})>1$).

\begin{theorem}\label{thm3} Suppose $(C1)$ - $(C3)$. Then there exist $\overline{c}^*<0<\overline{c}^{**}$ such that 
	\begin{enumerate}
	\item[i)] For any $1 \le k\le \gamma(\mathcal{C}_{-A})$ and $c\in (\overline{c}^*,0)$ there exist $\lambda_{c,k}^-<0$ and $v_{c,k}\in \mathcal{C}_{-A}$ such that 
	\begin{equation*}
		\Phi_{\lambda_{c,k}^-}(v_{c,k})=c\ \ \mbox{and}\ \ 	\Phi_{\lambda_{c,k}^-}'(v_{c,k})=0.
	\end{equation*}
\item[ii)] For each $1 \le k\le \gamma(\mathcal{C}_{-A})$ the map $c \mapsto \lambda_{c,k}^-$,  is continuous and increasing in $(\overline{c}^*,0)$, and satisfies $\displaystyle \lim_{c\to 0^-}\lambda_{c,k}^-=0$.
\item[iii)] If $\gamma(\mathcal{C}_{-A})=\infty$ then  $(\lambda_{c,k}^-)$ is a nonincreasing unbounded sequence, i.e.  $\lambda_{c,k}^- \ge \lambda_{c,k+1}^- \to -\infty$ as $k\to \infty$, for any $c\in (\overline{c}^*,0)$.
\item[iv)] If $\gamma(\mathcal{C}_{-A})=\infty$, then for any $\lambda<0$ there exist sequences $(v_n)\subset \mathcal{C}_{-A}$, $(c_n)\subset (\overline{c}^*,0)$ and $(k_n)\subset \mathbb{N}$ such that $c_n\to 0$, $k_n\to \infty$ and
\begin{equation*}
	\lambda=\lambda_{c_n,k_n}^-, \ \	\Phi_{\lambda_{c_n,k_n}^-}(v_n)=c_n\ \ \mbox{and}\ \ 	\Phi_{\lambda_{c_n,k_n}^-}'(v_n)=0, \quad \mbox{for every } n.
	\end{equation*}
Moreover $v_n\to 0$ in $X$, so $(\lambda,0)$ is a bifurcation point for any $\lambda<0$.
\item[v)] For any $1 \le k\le \gamma(\mathcal{C}_{-A}\cap \mathcal{C}_B)$ and $c\in (\overline{c}^*,\overline{c}^{**})$ there exist $\lambda_{c,k}^+<0$ and $u_{c,k}\in \mathcal{C}_{-A}\cap\mathcal{C}_B$ such that
\begin{equation*}
		\Phi_{\lambda_{c,k}^+}(u_{c,k})=c\ \ \mbox{and}\ \ 	\Phi_{\lambda_{c,k}^+}'(u_{c,k})=0.
	\end{equation*}
Moreover $\lambda_{c,k}^-<\lambda_{c,k}^+$ for all $c\in (\overline{c}^*,0)$.
\item[vi)] For any $1 \le k\le \gamma(\mathcal{C}_{-A}\cap \mathcal{C}_B)$ the map $c \mapsto \lambda_{c,k}^+$ is continuous and increasing in $(\overline{c}^*,\overline{c}^{**})$.
\item[vii)] If $\gamma(\mathcal{C}_{-A}\cap \mathcal{C}_B)=\infty$ then  $(\lambda_{c,k}^+)$ is a nonincreasing unbounded sequence, i.e.  $\lambda_{c,k}^+ \ge \lambda_{c,k+1}^+ \to -\infty$ as $k\to \infty$, for any $c\in (\overline{c}^*,\overline{c}^{**})$.
\end{enumerate}

\end{theorem}

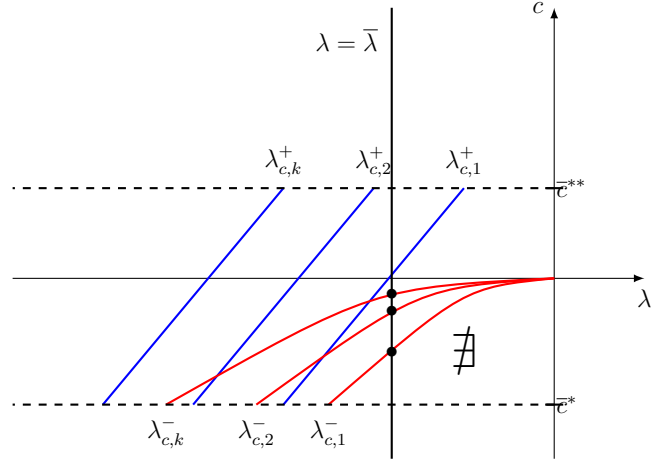
\begin{figure}[H]
	\centering
	\begin{tikzpicture}[>=latex,scale=1.2]
		\draw[->] (-6,0) -- (1,0) node[below] {\scalebox{0.8}{$\lambda$}};
		\draw[->] (0,-2) -- (0,3) node[left] {\scalebox{0.8}{$c$}};
		
		\draw[blue,thick] (-1,1) .. controls (-3,-1.4) .. (-3,-1.4);
		\draw[blue,thick] (-2,1) .. controls (-4,-1.4) .. (-4,-1.4);
		\draw[blue,thick] (-3,1) .. controls (-5,-1.4) .. (-5,-1.4);
		
		\draw[red,thick] (0,0) .. controls (-1,-.1) .. (-2.5,-1.4);
		\draw[red,thick] (0,0) .. controls (-1.5,-.1) .. (-3.3,-1.4);
		\draw[red,thick] (0,0) .. controls (-2,-0.1) .. (-4.3,-1.4);
		
		\draw [thick] (-1.8,-2) -- (-1.8,3);
		
		\draw (-1.8,-.18) node{\scalebox{0.8}{$\bullet$}};
		\draw (-1.8,-.37) node{\scalebox{0.8}{$\bullet$}};
		\draw (-1.8,-.82) node{\scalebox{0.8}{$\bullet$}};
		
		\draw  (-1,-.5) node[below]{\scalebox{1.5}{$\nexists$}}; 	
		
		\draw [thick] (-.1,-1.4) node[right]{\scalebox{0.8}{$\overline{c}^*$}} -- (.1,-1.4); 
		\draw [thick] (-.1,1) node[right]{\scalebox{0.8}{$\overline{c}^{**}$}} -- (.1,1); 
		\draw [thick,dashed] (0,-1.4) -- (-6,-1.4);
		\draw [thick,dashed] (0,1) -- (-6,1);
		
		\draw  (-2.3,2.4) node[above]{\scalebox{0.8}{$\lambda=\overline{\lambda}$}} ;
		
			\draw  (-1,1) node[above]{\scalebox{0.8}{$\lambda_{c,1}^+$}} ; 	
		\draw  (-2,1) node[above]{\scalebox{0.8}{$\lambda_{c,2}^+$}} ; 
		\draw   (-3,1) node[above]{\scalebox{0.8}{$\lambda_{c,k}^+$}} ; 
			\draw (-2.5,-2) node[above]{\scalebox{0.8}{$\lambda_{c,1}^-$}} ; 	
		\draw  (-3.3,-2) node[above]{\scalebox{0.8}{$\lambda_{c,2}^-$}} ; 
		\draw   (-4.3,-2) node[above]{\scalebox{0.8}{$\lambda_{c,k}^-$}} ; 
	\end{tikzpicture}
\caption{Energy curves for Theorem \ref{thm3}. Red curves correspond to $(\lambda_{c,k}^-,c)$, with $c\in (\overline{c}^*,0)$. Blue curves correspond to $(\lambda_{c,k}^+,c)$, with $c\in (\overline{c}^*,\overline{c}^{**})$. }\label{fig:CC2}
\end{figure}

\begin{remark}It is worth emphasizing that the functions $\lambda_{c,k}^+$ and $\lambda_{c,k}^-$ 
	appearing in Theorem \ref{thm2} also depend on the cone $\mathcal{C}_A$. 
	A more precise notation would therefore be 
	$\lambda_{c,k,\mathcal{C}_A}^+$ and $\lambda_{c,k,\mathcal{C}_A}^-$.
	For the sake of readability, however, we will avoid this heavier notation. 
	Furthermore, it follows that the functions 
	$\lambda_{c,k}^+$ and $\lambda_{c,k}^-$ in Theorem~\ref{thm3} 
	are entirely different from those in Theorem \ref{thm2}. 
	The similarity in notation is due solely to the fact that we are restricting the functional 
	to $\mathcal{N}_c^+$ or $\mathcal{N}_c^-$, which, strictly speaking, 
	should also depend on the cone. From now on, we will always assume 
	the reader is aware of this dependence.
	
\end{remark}

Let us additionally assume that $A$ and $B$ are related as follows:

\begin{enumerate}
	\item[(C4)] If $(u_n)\subset \mathcal{C}_A$ is a bounded sequence satisfying $A(u_n)\to 0$ then $B(u_n)\to 0$. 
\end{enumerate}

This condition allows us to take $c^{**}=\infty$, and to obtain a  result similar to \cite[Theorem 5.14]{MR4736027}:

\begin{theorem}\label{thm4} Suppose $(C1)$-$(C4)$. Then there exist $c^*<0$ such that: 
	\begin{enumerate}
		\item[i)] For any $1 \le k\le \gamma(\mathcal{C}_A)$ and $c\in (c^*,0)$ there exist $\lambda_{c,k}^+>0$ and $v_{c,k}\in \mathcal{C}_A$ such that 
		\begin{equation*}
			\Phi_{\lambda_{c,k}^+}(v_{c,k})=c\ \ \mbox{and}\ \ 	\Phi_{\lambda_{c,k}^+}'(v_{c,k})=0.
		\end{equation*}
\item[ii)] For any $1 \le k\le \gamma(\mathcal{C}_A)$ the map $c \mapsto \lambda_{c,k}^+$  is continuous and decreasing in $ (c^*,0)$, and satisfies $\displaystyle \lim_{c\to 0^-}\lambda_{c,k}=0$.
\item[iii)] If $\gamma(\mathcal{C}_A)=\infty$ then  $(\lambda_{c,k}^+)$ is a nondecreasing unbounded sequence, i.e. $0< \lambda_{c,k}^+ \le \lambda_{c,k+1}^+ \to \infty$ as $k\to \infty$, for any $c\in (c^*,0)$.
\item[iv)] If $\gamma(\mathcal{C}_A)=\infty$ then for each $\lambda>0$ there exist sequences $(v_n)\subset \mathcal{C}_A$, $(c_n)\subset (c^*,0)$ and $(k_n)\subset \mathbb{N}$ such that $c_n\to 0$, $k_n\to \infty$,  and
\begin{equation*}
\lambda=\lambda_{c_n,k_n}^+, \ \		\Phi_{\lambda_{c_n,k_n}^+}(v_n)=c_n\ \ \mbox{and}\ \ 	\Phi_{\lambda_{c_n,k_n}^+}'(v_n)=0, \quad \mbox{for every } n.
\end{equation*}
Moreover, $v_n\to 0$ in $X$, so $(\lambda,0)$ is a bifurcation point for any $\lambda>0$.
\item[v)] For any $1 \le k\le \gamma(\mathcal{C}_A\cap \mathcal{C}_B)$ and $c>c^*$ there exist $\lambda_{c,k}^-\in \mathbb{R}$ and $u_{c,k}\in \mathcal{C}_A\cap\mathcal{C}_B$ such that
\begin{equation*}
			\Phi_{\lambda_{c,k}^-}(u_{c,k})=c\ \ \mbox{and}\ \ 	\Phi_{\lambda_{c,k}^-}'(u_{c,k})=0.
		\end{equation*}
Moreover $\lambda_{c,k}^+<\lambda_{c,k}^-$ for all $c\in (c^*,0)$.
\item[vi)] For any $1 \le k\le \gamma(\mathcal{C}_A\cap \mathcal{C}_B)$ the map $c \mapsto \lambda_{c,k}^-$ is continuous and decreasing in $(c^*,\infty)$, and satisfies $\displaystyle \lim_{c\to \infty}\lambda_{c,k}^-=-\infty$.
\item[vii)] If $\gamma(\mathcal{C}_A\cap \mathcal{C}_B)=\infty$ then $(\lambda_{c,k}^-)$ is a nondecreasing unbounded sequence, , i.e. $0< \lambda_{c,k}^- \le \lambda_{c,k+1}^- \to \infty$ as $k\to \infty$, for any  $c\in (c^*,c^{**})$.
\item[viii)] If $\gamma(\mathcal{C}_A\cap \mathcal{C}_B)=\infty$, then for each $\lambda\in\mathbb{R}$ there exist sequences $(u_n)\subset \mathcal{C}_A\cap \mathcal{C}_B$, $(c_n) \subset (c^*,\infty)$ and $(k_n) \subset \mathbb{N}$ such that $c_n\to \infty$, $k_n\to \infty$,  and
\begin{equation*}
\lambda=\lambda_{c_n,k_n}^-, \ \	\Phi_{\lambda_{c_n,k_n}^-}(u_n)=c_n\ \ \mbox{and}\ \ 	\Phi_{\lambda_{c_n,k_n}^-}'(u_n)=0, \quad \mbox{for every } n.
\end{equation*}
Moreover $\|u_n\|\to \infty$, so $(\lambda,\infty)$ is a bifurcation point for any $\lambda\in\mathbb{R}$.
\end{enumerate}

\end{theorem}

\begin{figure}[H]
	\centering
	\begin{tikzpicture}[>=latex,scale=1.2]
		\draw[->] (-1,0) -- (6,0) node[below] {\scalebox{0.8}{$\lambda$}};
		\foreach \x in {}
		\draw[shift={(\x,0)}] (0pt,2pt) -- (0pt,-2pt) node[below] {\footnotesize $\x$};
		\draw[->] (0,-2) -- (0,3) node[left] {\scalebox{0.8}{$c$}};
		\foreach \y in {}
		\draw[shift={(0,\y)}] (2pt,0pt) -- (-2pt,0pt) node[left] {\footnotesize $\y$};
		\draw[blue,thick] (-2.5,1.2) .. controls (0,1) and (1,1) .. (2,0);
		\draw[blue,thick] (2,0) .. controls (3,-1.4) .. (3,-1.4);
		\draw[blue,thick] (-2.5,1.7) .. controls (0,1.5) and (1.5,1.5) .. (3,0);
		\draw[blue,thick] (3,0) .. controls (4,-1.4) .. (4,-1.4);
		
		\draw [thick] (3.3,0) node[above]{$\cdots$} -- (3.3,0); 
		\draw[blue,thick] (-2.5,2.2) .. controls (0,2) and (2,2) .. (4,0);
		\draw[blue,thick] (4,0) .. controls (5,-1.4) .. (5,-1.4);
		\draw[red,thick] (0,0) .. controls (1,-.1) .. (2.5,-1.4);
		\draw[red,thick] (0,0) .. controls (1.5,-.1) .. (3.3,-1.4);
		\draw[red,thick] (0,0) .. controls (2,-0.1) .. (4.3,-1.4);
		\draw [thick] (3.2,-1.3) node[above]{$\vdots$} -- (3.2,-1.3); 
		\draw [thick] (-.1,1.85) node[right]{$\vdots$} -- (-.1,1.85);
		\draw [thick] (1.8,-2) -- (1.8,3);
		\draw (1.8,.18) node{\scalebox{0.8}{$\bullet$}};
		\draw (1.8,.91) node{\scalebox{0.8}{$\bullet$}};
		\draw (1.8,1.50) node{\scalebox{0.8}{$\bullet$}};
		\draw (1.8,-.18) node{\scalebox{0.8}{$\bullet$}};
		\draw (1.8,-.37) node{\scalebox{0.8}{$\bullet$}};
		\draw (1.8,-.82) node{\scalebox{0.8}{$\bullet$}};
		\draw  (-2.8,1) node[above]{\scalebox{0.8}{$\lambda_{c,1}^-$}} ; 	
		\draw  (-2.8,1.5) node[above]{\scalebox{0.8}{$\lambda_{c,2}^-$}} ; 
		\draw  (-2.8,2) node[above]{\scalebox{0.8}{$\lambda_{c,k}^-$}} ; 
		\draw  (1,-.5) node[below]{\scalebox{1.5}{$\nexists$}}  ; 	
		\draw  (2.5,-2) node[above]{\scalebox{0.8}{$\lambda_{c,1}^+$}} ; 	
		\draw  (3.3,-2) node[above]{\scalebox{0.8}{$\lambda_{c,2}^+$}} ; 
		\draw  (4.3,-2) node[above]{\scalebox{0.8}{$\lambda_{c,k}^+$}} ; 
		\draw [thick] (-.1,-1.4) node[left]{\scalebox{0.8}{$c^*$}} -- (.1,-1.4); 
		\draw [thick,dashed] (0,-1.4) -- (6,-1.4);
		\draw  (2.3,2.4) node[above]{\scalebox{0.8}{$\lambda=\overline{\lambda}$}} ;
	\end{tikzpicture}
\caption{Energy curves for Theorem \ref{thm4}. Red curves corresponds to $(\lambda_{c,k}^+,c)$, $c\in (c^*,0)$ and blue curves are $(\lambda_{c,k}^-,c)$, $c\in (c^*,\infty)$. }\label{fig:CC3}
\end{figure}
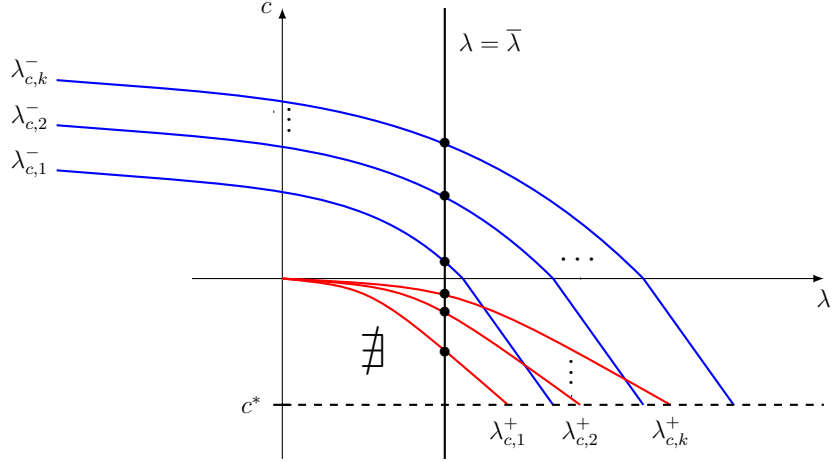


We also have the following counterpart of Theorem \ref{thm4}, which is obtained by considering $\mathcal{C}_{-A}$ instead of $\mathcal{C}_{A}$ :
\begin{theorem}\label{thm5} Suppose $(C1)$ - $(C4)$. Then there exists $\overline{c}^*<0$ such that:
		\begin{enumerate}
		\item[i)] For any $1 \le k\le \gamma(\mathcal{C}_{-A})$ and $c\in (\overline{c}^*,0)$ there exist $\lambda_{c,k}^-<0$ and $v_{c,k}\in \mathcal{C}_{-A}$ such that 
		\begin{equation*}
		\Phi_{\lambda_{c,k}^-}(v_{c,k})=c\ \ \mbox{and}\ \ 	\Phi_{\lambda_{c,k}^-}'(v_{c,k})=0.
		\end{equation*}
		\item[ii)] For any $1 \le k\le \gamma(\mathcal{C}_{-A})$ the map $c \mapsto \lambda_{c,k}^-$,  is continuous and increasing in $(\overline{c}^*,0)$, and satisfies $\displaystyle \lim_{c\to 0^-}\lambda_{c,k}^-=0$.
		\item[iii)] If $\gamma(\mathcal{C}_{-A})=\infty$ then  $\displaystyle \lim_{k\to \infty}\lambda_{c,k}^-=-\infty$ for each $c\in (\overline{c}^*,0)$.
		\item[iv)] If $\gamma(\mathcal{C}_{-A})=\infty$, then for each $\lambda<0$ there exist sequences $(v_n)\subset \mathcal{C}_{-A}$, $(c_n)\subset (\overline{c}^*,0)$ and $(k_n)\subset \mathbb{N}$ such that $c_n\to 0$, $k_n\to \infty$ and
		\begin{equation*}
		\lambda=\lambda_{c_n,k_n}^-, \ \	\Phi_{\lambda_{c_n,k_n}^-}(v_n)=c_n\ \ \mbox{and}\ \ 	\Phi_{\lambda_{c_n,k_n}^-}'(v_n)=0, \quad \mbox{for every } n.
		\end{equation*}
		Moreover $v_n\to 0$ in $X$, so $(\lambda,0)$ is a bifurcation point for any $\lambda<0$.
\item[v)] For any $1 \le k\le \gamma(\mathcal{C}_{-A}\cap \mathcal{C}_B)$ and $c> \overline{c}^*$ there exist $\lambda_{c,k}^+\in \mathbb{R}$ and $u_{c,k}\in \mathcal{C}_{-A}\cap\mathcal{C}_B$ such that
\begin{equation*}
			\Phi_{\lambda_{c,k}^+}(u_{c,k})=c\ \ \mbox{and}\ \ 	\Phi_{\lambda_{c,k}^+}'(u_{c,k})=0.
		\end{equation*}
Moreover $\lambda_{c,k}^+<\lambda_{c,k}^-$ for any $c\in (\overline{c}^*,0)$.
\item[vi)] For any $1 \le k\le \gamma(\mathcal{C}_{-A}\cap \mathcal{C}_B)$ the map $c \mapsto \lambda_{c,k}^+$ is continuous and increasing  in $(\overline{c}^*,\infty)$, and satisfies $\displaystyle \lim_{c\to \infty}\lambda_{c,k}^+=\infty$.
\item[vii)] If $\gamma(\mathcal{C}_{-A}\cap \mathcal{C}_B)=\infty$ then  $\displaystyle \lim_{k\to \infty}\lambda_{c,k}^+=-\infty$ for any $c>\overline{c}^*$.
\item[viii)] If $\gamma(\mathcal{C}_{-A}\cap \mathcal{C}_B)=\infty$ then for any $\lambda\in\mathbb{R}$ there exist sequences $( u_n)\subset \mathcal{C}_{-A}\cap \mathcal{C}_B$, $(c_n)\subset (c^*,\infty)$ and $(k_n)\subset \mathbb{N}$ such that $c_n\to \infty$, $k_n\to \infty$,  and
\begin{equation*}
\lambda=\lambda_{c_n,k_n}^+, \ \		\Phi_{\lambda_{c_n,k_n}^+}(u_n)=c_n\ \ \mbox{and}\ \ 	\Phi_{\lambda_{c_n,k_n}^+}'(u_n)=0, \quad \mbox{for every } n.
\end{equation*}
Moreover $\|u_n\|\to \infty$, so $(\lambda,\infty)$ is a bifurcation point for any $\lambda\in\mathbb{R}$.
\end{enumerate}

\end{theorem}

\begin{figure}[h]
	\centering
	\begin{tikzpicture}[>=latex,scale=1.2]
		\draw[->] (-6,0) -- (1,0) node[below] {\scalebox{0.8}{$\lambda$}};
		\foreach \x in {}
		\draw[shift={(\x,0)}] (0pt,2pt) -- (0pt,-2pt) node[below] {\footnotesize $\x$};
		\draw[->] (0,-2) -- (0,3) node[left] {\scalebox{0.8}{$c$}};
		\foreach \y in {}
		\draw[shift={(0,\y)}] (2pt,0pt) -- (-2pt,0pt) node[left] {\footnotesize $\y$};
		
		\draw[blue,thick] (2.5,1.2) .. controls (0,1) and (-1,1) .. (-2,0);
		\draw[blue,thick] (-2,0) .. controls (-3,-1.4) .. (-3,-1.4);
		
		\draw[blue,thick] (2.5,1.7) .. controls (0,1.5) and (-1.5,1.5) .. (-3,0);
		\draw[blue,thick] (-3,0) .. controls (-4,-1.4) .. (-4,-1.4);
		
		\draw[blue,thick] (2.5,2.2) .. controls (0,2) and (-2,2) .. (-4,0);
		\draw[blue,thick] (-4,0) .. controls (-5,-1.4) .. (-5,-1.4);
		
		\draw[red,thick] (0,0) .. controls (-1,-.1) .. (-2.5,-1.4);
		\draw[red,thick] (0,0) .. controls (-1.5,-.1) .. (-3.3,-1.4);
		\draw[red,thick] (0,0) .. controls (-2,-0.1) .. (-4.3,-1.4);
		
		\draw [thick] (-3.3,0) node[above]{$\cdots$} -- (-3.3,0); 
		\draw [thick] (-3.2,-1.3) node[above]{$\vdots$} -- (-3.2,-1.3); 
		\draw [thick] (.1,1.85) node[left]{$\vdots$} -- (.1,1.85);
		
		\draw [thick] (-1.8,-2) -- (-1.8,3);
		\draw (-1.8,.18) node{\scalebox{0.8}{$\bullet$}};
		\draw (-1.8,.91) node{\scalebox{0.8}{$\bullet$}};
		\draw (-1.8,1.50) node{\scalebox{0.8}{$\bullet$}};
		\draw (-1.8,-.18) node{\scalebox{0.8}{$\bullet$}};
		\draw (-1.8,-.37) node{\scalebox{0.8}{$\bullet$}};
		\draw (-1.8,-.82) node{\scalebox{0.8}{$\bullet$}};
		
		\draw  (2.8,1) node[above]{\scalebox{0.8}{$\lambda_{c,1}^+$}} ; 	
		\draw  (2.8,1.5) node[above]{\scalebox{0.8}{$\lambda_{c,2}^+$}} ; 
		\draw  (2.8,2) node[above]{\scalebox{0.8}{$\lambda_{c,k}^+$}} ; 
		\draw  (-1,-.5) node[below]{\scalebox{1.5}{$\nexists$}}  ; 	
		\draw  (-2.5,-2) node[above]{\scalebox{0.8}{$\lambda_{c,1}^-$}} ; 	
		\draw  (-3.3,-2) node[above]{\scalebox{0.8}{$\lambda_{c,2}^-$}} ; 
		\draw  (-4.3,-2) node[above]{\scalebox{0.8}{$\lambda_{c,k}^-$}} ; 
		
		\draw [thick] (.1,-1.4) node[right]{\scalebox{0.8}{$\overline{c}^*$}} -- (-.1,-1.4); 
		\draw [thick,dashed] (0,-1.4) -- (-6,-1.4);
		\draw  (-2.3,2.4) node[above]{\scalebox{0.8}{$\lambda=\overline{\lambda}$}} ;
	\end{tikzpicture}
\caption{Energy curves for Theorem \ref{thm5}. Red curves corresponds to $(\lambda_{c,k}^-,c)$, $c\in (c^*,0)$ and blue curves are $(\lambda_{c,k}^+,c)$, $c\in (c^*,\infty)$. }\label{fig:CC4}
\end{figure}
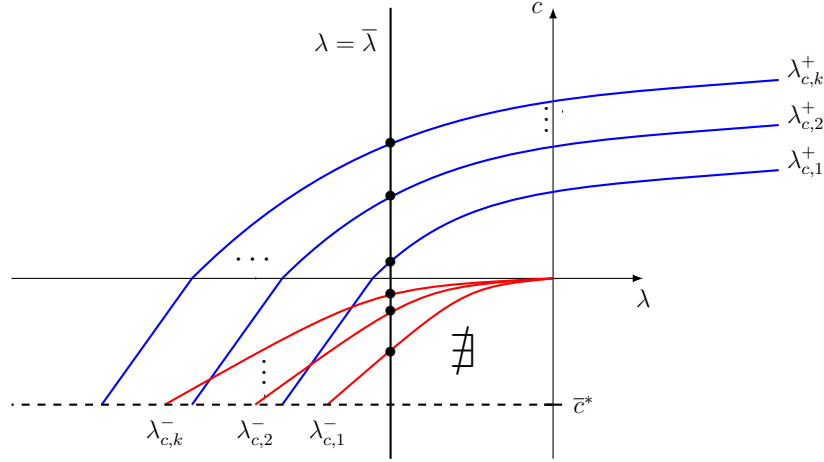

Let us conclude by observing that Theorems \ref{thm2} and \ref{thm3} (as well as Theorem \ref{thm4} and \ref{thm5}) both apply if, for instance, $\gamma(\mathcal{C}_{A}\cap \mathcal{C}_B)>1$ and $\gamma(\mathcal{C}_{-A}\cap \mathcal{C}_B)>1$. In particular, this happens if $\gamma(\mathcal{C}_{A}\cap \mathcal{C}_B)=\gamma(\mathcal{C}_{-A}\cap \mathcal{C}_B)=\infty$, and in such case a superposition of Figures \ref{fig:CC} and \ref{fig:CC2}, or Figures \ref{fig:CC3} and \ref{fig:CC4} would describe our results.
\subsection{Applications}

In our first application we consider the problem
\begin{equation}
	\label{quasi}
	\begin{cases}
		-\Delta_p u=\lambda a(x)|u|^{\alpha-2}u +b(x)|u|^{\beta-2}u &\mbox{ in } \Omega, \\
		u=0 &\mbox{ on } \partial \Omega,
	\end{cases}
\end{equation}
where $1<\alpha <p<\beta<p^*$ and $\Omega$ is a bounded domain of $\mathbb{R}^N$. We assume, for simplicity, that $a,b\in L^{\infty}(\Omega)$. The corresponding functional is given by \eqref{dephi} with
\begin{equation*}
	N(u)=\int_\Omega |\nabla u|^pdx,\ \ \ A(u)=\int_\Omega a(x)|u|^\alpha dx,\ \ \ B(u)=\int_\Omega b(x)|u|^\beta dx,\mbox{ for } u\in X=W_0^{1,p}(\Omega).
\end{equation*}
 Note that $\eta=p$. We set
\begin{equation*}
	\mathcal{A}^{+}:=\{x \in \Omega:\ a(x)>0 \}, 
\end{equation*}
i.e. $\mathcal{A}^{+}$ is the largest open subset of $\Omega$ where $a>0$ a.e. In a similar way we set
\begin{equation*}
 \mathcal{A}^{-}:=\{x\in \Omega:\ a(x)<0 \},\quad \mbox{and} \quad \mathcal{B}^{+}:=\{x\in \Omega:\ b(x)>0 \}.
\end{equation*}
We also set $ \mathcal{A}^{0}:= \Omega \setminus (\mathcal{A}^{+} \cup \mathcal{A}^{-})$.
Our main result on \eqref{quasi} reads as follows:

\begin{theorem}\label{thmapp1} Suppose that $\mathcal{A}^+\cap \mathcal{B}^+ \neq \emptyset$.  Then there exist $c^*<0<c^{**}$ such that:
	\begin{enumerate}
	\item[i)] For any $c\in (c^*,0)$ there exist sequences $(\lambda_{c,k}^+) \subset \mathbb{R}$ and $(v_{c,k})\subset \mathcal{C}_A$   such that 
	\begin{equation*}
		\Phi_{\lambda_{c,k}^+}(v_{c,k})=c\ \ \mbox{and}\ \ 	\Phi_{\lambda_{c,k}^+}'(v_{c,k})=0,
	\end{equation*}
	i.e. $v_{c,k}$ is a weak solution of \eqref{quasi} with $\lambda=\lambda_{c,k}^+$, for any $k \in \mathbb{N}$.
Moreover:
\begin{enumerate}
\item For any $c\in (c^*,0)$ the sequence $(\lambda_{c,k}^+)$ is positive, nondecreasing and unbounded, i.e. $0< \lambda_{c,k}^+ \le \lambda_{c,k+1}^+ \to \infty$ as $k\to \infty$, 
\item  For any $k \in \mathbb{N}$ the map $c \mapsto \lambda_{c,k}^+$ is continuous and decreasing in $(c^*,0)$, and $\displaystyle \lim_{c\to 0^-}\lambda_{c,k}^+=0$.
\item For any $\lambda>0$ there exist sequences $(v_n)\subset \mathcal{C}_A$, $(c_n)\subset (c^*,0)$ and $(k_n)\subset \mathbb{N}$ such that $c_n\to 0$, $k_n\to \infty$ and \begin{equation*}
\lambda=\lambda_{c_n,k_n}^+, \ \		\Phi_{\lambda_{c_n,k_n}^+}(v_n)=c_n\ \ \mbox{and}\ \ 	\Phi_{\lambda_{c_n,k_n}^+}'(v_n)=0, \quad \mbox{for every } n.
\end{equation*}
Furthermore $v_n\to 0$ in $X$, so $(\lambda,0)$ is a bifurcation point for any $\lambda>0$.
\end{enumerate}

\item[ii)] For any $c\in (c^*,c^{**})$ there exist sequences $(\lambda_{c,k}^-) \subset \mathbb{R}$ and $(u_{c,k})\subset \mathcal{C}_A\cap\mathcal{C}_B$   such that 
\begin{equation*}
\Phi_{\lambda_{c,k}^-}(u_{c,k})=c\ \ \mbox{and}\ \ 	\Phi_{\lambda_{c,k}^-}'(u_{c,k})=0,
\end{equation*}
i.e. $u_{c,k}$ is a weak solution of \eqref{quasi} with $\lambda=\lambda_{c,k}^-$, for any $k \in \mathbb{N}$.
Moreover:
\begin{enumerate}
	\item For any $c\in (c^*,0)$ the sequence $(\lambda_{c,k}^-)$ is positive, nondecreasing and unbounded, i.e. $0< \lambda_{c,k}^- \le \lambda_{c,k+1}^- \to \infty$ as $k\to \infty$, and $\lambda_{c,k}^+<\lambda_{c,k}^-$ for every $k \in \mathbb{N}$.
	\item  For any $k \in \mathbb{N}$ the map $c \mapsto \lambda_{c,k}^-$ is continuous and decreasing in $(c^*,c^{**})$.
	\item If $a \geq 0$ and $\mathcal{A}^0 \subset \mathcal{B}^0$ then the previous assertions hold with $c^{**}=\infty$. In addition, for any $\lambda>0$ there exist sequences $(u_n)\subset \mathcal{C}_A$, $(c_n)\subset (c^*,0)$ and $(k_n)\subset \mathbb{N}$ such that $c_n\to 0$, $k_n\to \infty$ and \begin{equation*}
	\lambda=\lambda_{c_n,k_n}^-, \ \		\Phi_{\lambda_{c_n,k_n}^-}(u_n)=c_n\ \ \mbox{and}\ \ 	\Phi_{\lambda_{c_n,k_n}^-}'(u_n)=0, \quad \mbox{for every } n.
	\end{equation*}
	Furthermore $\|u_n\|\to \infty$, so $(\lambda,\infty)$ is a bifurcation point for any $\lambda\in\mathbb{R}$.
\end{enumerate}

\end{enumerate}
\end{theorem}

 Theorem \ref{thmapp1} has the following counterpart when dealing with $\mathcal{A}^- \cap \mathcal{B}^+$:
 
\begin{theorem}\label{thmapp1.1} Suppose that  $\mathcal{A}^-\cap \mathcal{B}^+ \neq \emptyset$. Then there exist $\overline{c}^*<0<\overline{c}^{**}$ such that:
	\begin{enumerate}
		\item[i)] For any $c\in (\overline{c}^*,0)$ there exist sequences $(\mu_{c,k}^-) \subset \mathbb{R}$ and $(v_{c,k})\subset \mathcal{C}_{-A}$   such that 
		\begin{equation*}
		\Phi_{\mu_{c,k}^-}(v_{c,k})=c\ \ \mbox{and}\ \ 	\Phi_{\mu_{c,k}^-}'(v_{c,k})=0,
		\end{equation*}
		i.e. $v_{c,k}$ is a weak solution of \eqref{quasi} with $\lambda=\mu_{c,k}^-$, for any $k \in \mathbb{N}$.
		Moreover:
		\begin{enumerate}
			\item For any $c\in (\overline{c}^*,0)$ the sequence $(\mu_{c,k}^-)$ is negative, nonincreasing and unbounded, i.e. $0> \mu_{c,k}^- \ge \mu_{c,k+1}^- \to -\infty$ as $k\to \infty$. 
			\item  For any $k \in \mathbb{N}$ the map $c \mapsto \mu_{c,k}^-$ is continuous and increasing in $(\overline{c}^*,0)$, and $\displaystyle \lim_{c\to 0^-}\mu_{c,k}^-=0$.
			\item For any $\lambda<0$ there exist sequences $(v_n)\subset \mathcal{C}_A$, $(c_n)\subset (c^*,0)$ and $(k_n)\subset \mathbb{N}$ such that $c_n\to 0$, $k_n\to \infty$ and \begin{equation*}
			\lambda=\mu_{c_n,k_n}^-, \ \		\Phi_{\mu_{c_n,k_n}^-}(v_n)=c_n\ \ \mbox{and}\ \ 	\Phi_{\mu_{c_n,k_n}^-}'(v_n)=0, \quad \mbox{for every } n.
			\end{equation*}
			Furthermore $v_n\to 0$ in $X$, so $(\lambda,0)$ is a bifurcation point for any $\lambda<0$.
		\end{enumerate}
		
\item[ii)] For any $c\in (\overline{c}^*,\overline{c}^{**})$ there exist sequences $(\mu_{c,k}^+) \subset \mathbb{R}$ and $(u_{c,k})\subset \mathcal{C}_{-A}\cap\mathcal{C}_B$   such that 
\begin{equation*}
\Phi_{\mu_{c,k}^+}(u_{c,k})=c\ \ \mbox{and}\ \ 	\Phi_{\mu_{c,k}^+}'(u_{c,k})=0,
\end{equation*}
i.e. $u_{c,k}$ is a weak solution of \eqref{quasi} with $\lambda=\mu_{c,k}^+$, for any $k \in \mathbb{N}$.
Moreover:
\begin{enumerate}
	\item For any $c\in (\overline{c}^*,\overline{c}^{**})$ the sequence $(\mu_{c,k}^+)$ is negative, nonincreasing and unbounded, i.e. $0> \mu_{c,k}^+ \ge \mu_{c,k+1}^+ \to -\infty$ as $k\to \infty$, and $\mu_{c,k}^+>\mu_{c,k}^-$ for every $k \in \mathbb{N}$.
	\item  For any $k \in \mathbb{N}$ the map $c \mapsto \mu_{c,k}^+$ is continuous and increasing in $(\overline{c}^*,\overline{c}^{**})$.
\end{enumerate}
\end{enumerate}
\end{theorem}

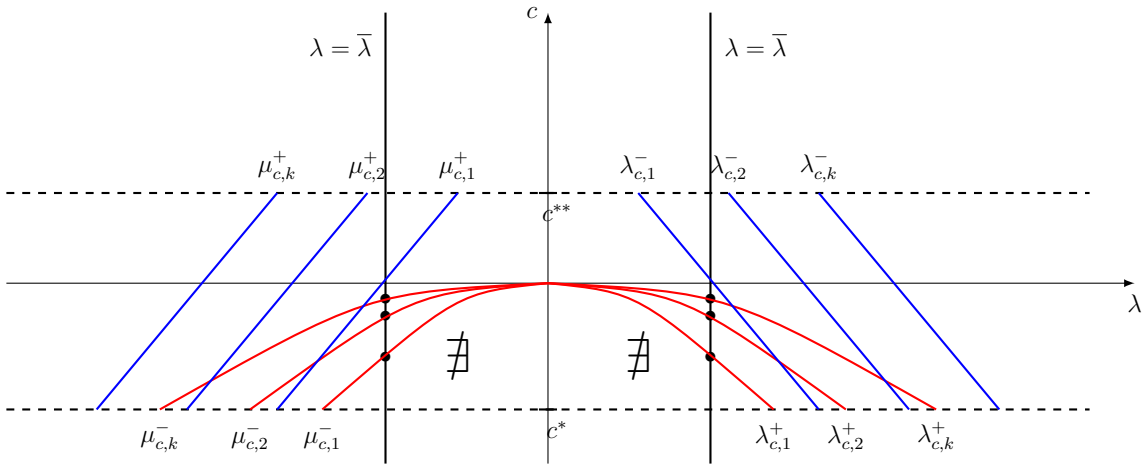
\begin{figure}[H]
	\centering
	\begin{tikzpicture}[>=latex,scale=1.2]
		\draw[->] (-6,0) -- (6.5,0) node[below] {\scalebox{0.8}{$\lambda$}};
		\draw[->] (0,-2) -- (0,3) node[left] {\scalebox{0.8}{$c$}};
		
		\draw [thick,dashed] (-6,-1.4) -- (6,-1.4);
		\draw [thick,dashed] (-6,1) -- (6,1);
		\draw [thick] (-.1,-1.4) -- (.1,-1.4) node[right,below]{\scalebox{0.8}{$c^*$}};
		\draw [thick] (-.1,1) -- (.1,1) node[right,below]{\scalebox{0.8}{$c^{**}$}};
		
		\draw [thick] (1.8,-2) -- (1.8,3);
		\draw [thick] (-1.8,-2) -- (-1.8,3);
		
		\foreach \y in {-0.18,-0.37,-0.82} {
			\draw (1.8,\y) node{\scalebox{0.8}{$\bullet$}};
			\draw (-1.8,\y) node{\scalebox{0.8}{$\bullet$}};
		}
		
		\draw (2.3,2.4) node[above]{\scalebox{0.8}{$\lambda=\overline{\lambda}$}};
		\draw (-2.3,2.4) node[above]{\scalebox{0.8}{$\lambda=\overline{\lambda}$}};
		
		\draw[red,thick] (0,0) .. controls (1,-.1) .. (2.5,-1.4);
		\draw[red,thick] (0,0) .. controls (1.5,-.1) .. (3.3,-1.4);
		\draw[red,thick] (0,0) .. controls (2,-0.1) .. (4.3,-1.4);
		
		\draw[red,thick] (0,0) .. controls (-1,-.1) .. (-2.5,-1.4);
		\draw[red,thick] (0,0) .. controls (-1.5,-.1) .. (-3.3,-1.4);
		\draw[red,thick] (0,0) .. controls (-2,-0.1) .. (-4.3,-1.4);
		
		\draw[blue,thick] (1,1) .. controls (3,-1.4) .. (3,-1.4);
		\draw[blue,thick] (2,1) .. controls (4,-1.4) .. (4,-1.4);
		\draw[blue,thick] (3,1) .. controls (5,-1.4) .. (5,-1.4);
		
		\draw[blue,thick] (-1,1) .. controls (-3,-1.4) .. (-3,-1.4);
		\draw[blue,thick] (-2,1) .. controls (-4,-1.4) .. (-4,-1.4);
		\draw[blue,thick] (-3,1) .. controls (-5,-1.4) .. (-5,-1.4);
		
		\draw (1,1) node[above]{\scalebox{0.8}{$\lambda_{c,1}^-$}};
		\draw (2,1) node[above]{\scalebox{0.8}{$\lambda_{c,2}^-$}};
		\draw (3,1) node[above]{\scalebox{0.8}{$\lambda_{c,k}^-$}};
		\draw (2.5,-2) node[above]{\scalebox{0.8}{$\lambda_{c,1}^+$}};
		\draw (3.3,-2) node[above]{\scalebox{0.8}{$\lambda_{c,2}^+$}};
		\draw (4.3,-2) node[above]{\scalebox{0.8}{$\lambda_{c,k}^+$}};
		
		\draw (-1,1) node[above]{\scalebox{0.8}{$\mu_{c,1}^+$}};
		\draw (-2,1) node[above]{\scalebox{0.8}{$\mu_{c,2}^+$}};
		\draw (-3,1) node[above]{\scalebox{0.8}{$\mu_{c,k}^+$}};
		\draw (-2.5,-2) node[above]{\scalebox{0.8}{$\mu_{c,1}^-$}};
		\draw (-3.3,-2) node[above]{\scalebox{0.8}{$\mu_{c,2}^-$}};
		\draw (-4.3,-2) node[above]{\scalebox{0.8}{$\mu_{c,k}^-$}};
		
		\draw (1,-.5) node[below]{\scalebox{1.5}{$\nexists$}};
		\draw (-1,-.5) node[below]{\scalebox{1.5}{$\nexists$}};
	\end{tikzpicture}
\caption{Energy curves from Theorems \ref{thmapp1} and \ref{thmapp1.1} under the conditions $\mathcal{A}^+\cap \mathcal{B}^+ \neq \emptyset$ and $\mathcal{A}^-\cap \mathcal{B}^+ \neq \emptyset$. We assume here that $c^*=\overline{c}^*$ and $c^{**}=\overline{c}^{**}$ for the sake of simplicity. }
\label{fig:combinedCC}
\end{figure}


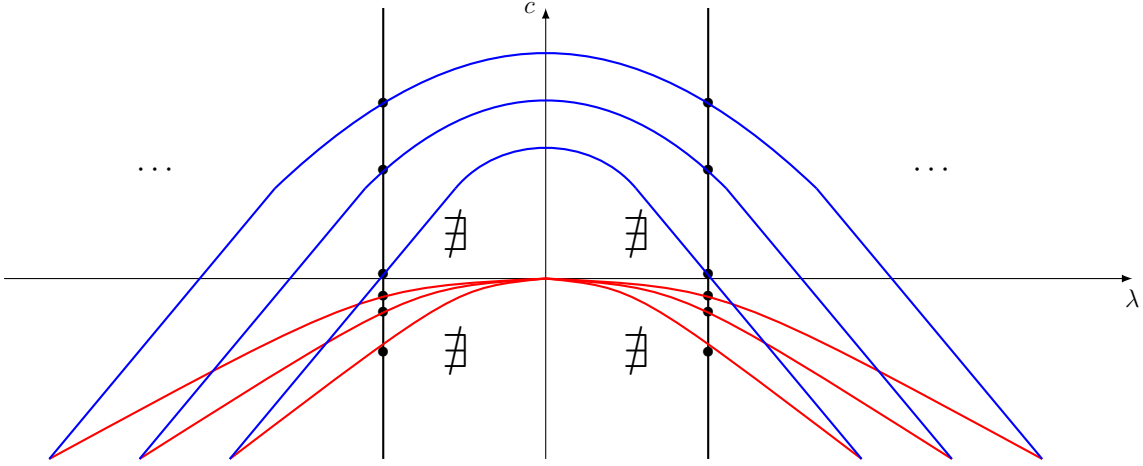
\begin{figure}[H]
	\centering
	\begin{tikzpicture}[>=latex,scale=1.2]
		\draw[->] (-6,0) -- (6.5,0) node[below] {\scalebox{0.8}{$\lambda$}};
		\draw[->] (0,-2) -- (0,3) node[left] {\scalebox{0.8}{$c$}};
		
		
		\draw [thick] (1.8,-2) -- (1.8,3);
		\draw [thick] (-1.8,-2) -- (-1.8,3);
		
		\foreach \y in {-0.20,-0.37,-0.82,.05,1.2,1.94} {
			\draw (1.8,\y) node{\scalebox{0.8}{$\bullet$}};
			\draw (-1.8,\y) node{\scalebox{0.8}{$\bullet$}};
		}
		
		\draw (4.3,1) node[above]{\scalebox{1}{$\cdots$}};
		\draw (-4.3,1) node[above]{\scalebox{1}{$\cdots$}};
		
		\draw[red,thick] (0,0) .. controls (1,-.1) .. (3.5,-2);
		\draw[red,thick] (0,0) .. controls (1.5,-.1) .. (4.5,-2);
		\draw[red,thick] (0,0) .. controls (2,-0.1) .. (5.5,-2);
		
		\draw[red,thick] (0,0) .. controls (-1,-.1) .. (-3.5,-2);
		\draw[red,thick] (0,0) .. controls (-1.5,-.1) .. (-4.5,-2);
		\draw[red,thick] (0,0) .. controls (-2,-0.1) .. (-5.5,-2);
		
		\draw[blue,thick] (1,1) .. controls (3,-1.4) .. (3.5,-2);
		\draw[blue,thick] (2,1) .. controls (4,-1.4) .. (4.5,-2);
		\draw[blue,thick] (3,1) .. controls (5,-1.4) .. (5.5,-2);
		
		\draw[blue,thick] (-1,1) .. controls (-0.5,1.6) and (0.5,1.6) .. (1,1);
		\draw[blue,thick] (-2,1) .. controls (-0.7,2.3) and (0.7,2.3) .. (2,1);
		\draw[blue,thick] (-3,1) .. controls (-0.9,3) and (0.9,3) .. (3,1);
		
		\draw[blue,thick] (-1,1) .. controls (-3,-1.4) .. (-3.5,-2);
		\draw[blue,thick] (-2,1) .. controls (-4,-1.4) .. (-4.5,-2);
		\draw[blue,thick] (-3,1) .. controls (-5,-1.4) .. (-5.5,-2);
		
		
		
		\draw (1,-.5) node[below]{\scalebox{1.5}{$\nexists$}};
		\draw (-1,-.5) node[below]{\scalebox{1.5}{$\nexists$}};
			\draw (1,.8) node[below]{\scalebox{1.5}{$\nexists$}};
		\draw (-1,.8) node[below]{\scalebox{1.5}{$\nexists$}};

	\end{tikzpicture}
\caption{Possible complete energy curves diagram to Theorem \ref{thmapp1.1}}
\label{fig:combinedCC1}
\end{figure}

\begin{conjecture}\label{conje} We believe that the curves in Figure \ref{fig:combinedCC} can be joined to produce a figure similar to Figure \ref{fig:combinedCC1}. In this case we would have, for any $\lambda\in \mathbb{R}$, the existence of two sequences of solutions (one with negative energy, the other one with positive energy) and, as a consequence, bifurcation from both $0$ and $\infty$ would occur. See Section ... for further discussion.
	
\end{conjecture}

To support our conjecture, we have the following result:
\begin{theorem}\label{thmapp1.1conje} Under the assumptions of Theorem \ref{thmapp1.1}, there exist $a\in L^\infty(\Omega)$ and $c^{***}>c^{**}$ such that $\lambda_{c,1}^-$ can be extended to $(c^*,c^{***})$, as a continuous and decreasing map. Moreover $\lambda_{c^{**},1}^-=0$ and $\lambda_{c,1}^-<0$ if $c\in (c^{**},c^{***})$.

\end{theorem}

Theorems \ref{thmapp1}, \ref{thmapp1.1} and \ref{thmapp1.1conje} improve the results in \cite{BrWu,BrWu1,Gu}. \\

Next we apply our results to the following problem:

\begin{equation} \label{103}
	\begin{cases}
		-\Delta_p u + |u|^{p-2}u = \lambda a(x)|u|^{\alpha-2}u+b(x)|u|^{\beta-2}u & \text{em } \mathbb{R}^N, \\
		\lim\limits_{|x|\to +\infty} u(x) = 0;
	\end{cases}
\end{equation}
where $p>N$, $1<\alpha<p<\beta$ and $a,b\in L^1(\mathbb{R}^N)$. We look for solutions of \eqref{103} in the standard Sobolev space $X:=W^{1,p}(\mathbb{R}^N)$ with norm given by
\begin{equation*}
	\|u\|=\left(\|u\|_p^p+\|\nabla u\|_p^p\right)^{\frac{1}{p}},\ u\in W^{1,p}(\mathbb{R}^N).
\end{equation*}
Now we have
\begin{equation*}
	N(u)=\int_{\mathbb{R}^N} |\nabla u|^pdx+\int_{\mathbb{R}^N} |u|^pdx,\ \ \ A(u)=\int_{\mathbb{R}^N} a(x)|u|^\alpha dx,\ \ \ B(u)=\int_{\mathbb{R}^N} a(x)|u|^\beta dx,\ u\in X.
\end{equation*}

\begin{theorem}\label{thmapp2} Suppose that the set  $A^+\cap B^+$ has an interior point, then there exists $c^*<0$ and $c^{**}>0$ such that
	\begin{enumerate}
		\item[i)] For all $c\in (c^*,0)$ there exist $\lambda_{c,k}^+>0$, $k\in \mathbb{N}$, with $\lambda_{c,k}\to \infty$ as $k\to \infty$, and $v_{c,k}\in \mathcal{C}_A$ such that 
		\begin{equation*}
			\Phi_{\lambda_{c,k}^+}(v_{c,k})=c\ \ \mbox{and}\ \ 	\Phi_{\lambda_{c,k}^+}'(v_{c,k})=0.
		\end{equation*}
Moreover, the function $\lambda_{c,k}^+$, $c\in (c^*,0)$ is continuous, decreasing and $\lim_{c\to 0^-}\lambda_{c,k}=0$.
\item[ii)] For each $\lambda>0$ we can find sequences $v_n\in \mathcal{C}_A$, $c_n\in (c^*,0)$ and $k_n\in \mathbb{N}$ such that $c_n\to 0$, $k_n\to \infty$ and
\begin{equation*}
			\Phi_{\lambda_{c_n,k_n}^+}(v_n)=c_n\ \ \mbox{and}\ \ 	\Phi_{\lambda_{c_n,k_n}^+}'(v_n)=0.
		\end{equation*}
Moreover, $v_n\to 0$ in $X$, so $(\lambda,0)$ is a bifurcation point for any $\lambda>0$.
\item[iii)] For all $c\in (c^*,c^{**})$ there exist $\lambda_{c,k}^->0$, $k\in \mathbb{N}$, with $\lambda_{c,k}^-\to \infty$ as $k\to \infty$, and $u_{c,k}\in \mathcal{C}_A\cap\mathcal{C}_B$ such that
\begin{equation*}
			\Phi_{\lambda_{c,k}^-}(u_{c,k})=c\ \ \mbox{and}\ \ 	\Phi_{\lambda_{c,k}^-}'(u_{c,k})=0.
		\end{equation*}
Moreover $\lambda_{c,k}^+<\lambda_{c,k}^-$ for all $c\in (c^*,0)$, and the function $\lambda_{c,k}^-$, $c\in (c^*,c^{**})$ is continuous and decreasing.
\end{enumerate}
If we assume furthermore that $a=b$, and $\mathcal{A}^-=\emptyset$, then $c^{**}$ can rreplaced by $\infty$ in $iii)$ and the following assertion holds:
\begin{enumerate}
		\item[iv)] For each $\lambda\in\mathbb{R}$ we can find sequences $u_n\in \mathcal{C}_A\cap \mathcal{C}_B$, $c_n\in (c^*,\infty)$ and $k_n\in \mathbb{N}$ such that $c_n\to \infty$, $k_n\to \infty$, $\lambda=\lambda_{c_n,k_n}^-$ and  
		\begin{equation*}
			\Phi_{\lambda_{c_n,k_n}^-}(u_n)=c_n\ \ \mbox{and}\ \ 	\Phi_{\lambda_{c_n,k_n}^-}'(u_n)=0.
		\end{equation*}
Moreover, $u_n\to \infty$ in $X$, so $(\lambda,\infty)$ is a bifurcation point for any $\lambda\in\mathbb{R}$.
\end{enumerate}
\end{theorem}

Theorem \ref{thmapp2} greatly improves \cite[Theorem 1.1]{AmBoPe}. In fact, aside from treating the case where $a,b$ can change sing, we see that in the case $a=b$ and $\mathcal{A}^-=\emptyset$, which is exactly the case contained in \cite[Theorem 1.1]{AmBoPe}, our results provides infinitely many solutions for all $\lambda\in \mathbb{R}$.

To conclude the applications let us mention that result analogous to equation \eqref{quasi} can be proved to the $p$-Fractional Laplacian equation 

\begin{equation}\label{105}
	\begin{cases}
		\displaystyle
		-2 \int_{\mathbb{R}^N} \frac{|u(y) - u(x)|^{p-2}(u(y) - u(x))}{|x - y|^{N + ps}} \, dy = \lambda a(x)|u|^{\alpha-2}u + b(x)|u|^{\beta-2}u & \text{in } \Omega, \\
		u = 0 & \text{in } \mathbb{R}^N \setminus \Omega.
	\end{cases}
\end{equation}
where $\Omega \subset \mathbb{R}^N$ is a bounded domain with smooth boundary, $s \in (0,1)$, $\lambda > 0$ is a parameter and $a,b\in L^\infty(\Omega)$. For the relevant definitions, we refer the reader to \cite{Gu}. Our results significantly improve upon those found therein.

\section{Proof of main results}
\subsection{Proof of Theorem \ref{thm1}}
In order to prove Theorem \ref{thm1} we will make use of \cite[Theorem 3.1]{Sz} (see also \cite[Section 4]{Sz} and the discussion there). We start with some technical results. First, let us show that condition $(CS)$ at the Appendix holds true, so we have a deformation lemma for non-complete spaces.

\begin{lemma}\label{CScondition} Suppose $(H1)$ and $(H2)$. Then condition $(CS)$ at the Appendix holds true with $X=\mathcal{S}_\mathcal{C}$, $d$ the Finsler metric and  $f(\cdot)=\widetilde{\Lambda}(c,\cdot)$.
\end{lemma}
\begin{proof} Suppose that $(u_n) \subset \mathcal{S}_\mathcal{C}$ is a Cauchy sequence with respect to the Finsler metric of $S_{\mathcal{C}}$. Then, as a sequence in $X\setminus\{0\}$ we have two possibilities: $u_n\to u\in \mathcal{S}_\mathcal{C}$ or $u_n\to  u\in \partial\mathcal{C}$. If $u_n$ does not converge in $\mathcal{S}_\mathcal{C}$, then the second possibility happens and, thanks to condition $(H2)$, we obtain that $\widetilde{\Lambda}(c,u_n)\to \infty$.
\end{proof}

Given $\lambda\in \mathbb{R}$ we denote by $K_\lambda$ the set o critical points of $\widetilde{\Lambda}(c,\cdot)$ at the level $\lambda$. By $(H2)$ this set is compact and then $\gamma (K_\lambda)$ is well defined. In the next result we will use the properties of genus contained in \cite[Proposition 2.3]{Sz}. See also the proof of \cite[Corollary 4.1]{Sz} and the properties of category in \cite[Proposition 2.2]{Sz}.

\begin{lemma}\label{genus} Suppose $(H1)$, $(H2)$ and $\widetilde{\Lambda}(c,\cdot)$ satisfies the Palais--Smale condition at any level. Fix $m\ge 2$ and suppose that $\gamma(\mathcal{S}_\mathcal{C})\ge m$. If $\lambda \in \mathbb{R}$, then there exists $\varepsilon>0$ such that 
	\begin{equation*}
		\mbox{if}\ \lambda-\varepsilon<\lambda_{c,k}\le\cdots\le \lambda_{c,k+m-1}<\lambda+\varepsilon\ \mbox{then}\ \gamma(K_\lambda)\ge m.
	\end{equation*}

\end{lemma}

\begin{proof} Indeed, there exists a neighborhood $N(K_\lambda)\in \F$ such that $\gamma(N(K_\lambda))=\gamma(K_\lambda)$. By Lemma \ref{CScondition} we can apply Theorem \ref{dl} to find a deformation $\eta\in C(\mathcal{S}_\mathcal{C},\mathcal{S}_\mathcal{C})$, which in our case can be assumed to be odd (see \cite{CoDeMa}), and is such that
	\begin{equation}\label{inclusub}
		\eta((\widetilde{\Lambda}(c,\cdot))^{\lambda+\varepsilon}\setminus N(K_\lambda))\subset ((\widetilde{\Lambda}(c,\cdot))^{\lambda-\varepsilon}.
	\end{equation}
Now suppose, on the contrary, that $\gamma(K_\lambda)<m$. Choose $M\in \F_{k+m-1}$ such that $\sup_{u\in M}\widetilde{\Lambda}(c,u)<\lambda+\varepsilon$. By \eqref{inclusub} it follows that
\begin{equation*}
	\eta(\overline{M\setminus N(K_\lambda)})\subset  ((\widetilde{\Lambda}(c,\cdot))^{\lambda-\varepsilon}.
\end{equation*}
Moreover
\begin{equation*}
	\gamma(	\eta(\overline{M\setminus N(K_\lambda)}))\ge \gamma(\overline{M\setminus N(K_\lambda)})\ge \gamma(M)-\gamma(K_\lambda)>k-1.
\end{equation*}
Since $	\eta(\overline{M\setminus N(K_\lambda)})$ is compact and symmetric, we conclude that $	\eta(\overline{M\setminus N(K_\lambda)})\in \F_k$. However, this contradicts the inequality
\begin{equation*}
 \lambda-\varepsilon<c_k\le \sup_{u\in \eta(\overline{M\setminus N(K_\lambda)})}\widetilde{\Lambda}(c,u)\le \lambda-\varepsilon, 
 \end{equation*}
and thus $\gamma(K_\lambda)\ge m$.

\end{proof}

We are now in position to prove Theorem \ref{thm1}:

\begin{proof}[Proof of Theorem \ref{thm1}] It suffices to show that if $c\in I$ and $k\le \gamma(\mathcal{S}_\mathcal{C})$, then $\lambda_{c,k}$ is a critical value of $\widetilde{\Lambda}(c,\cdot)$. We claim that, for all $\lambda\in \mathbb{R}$, the set $(\widetilde{\Lambda}(c,\cdot))^\lambda:=\{u\in S_{\mathcal{C}}:\ \widetilde{\Lambda}(c,u)\le \lambda \}$ is complete. Indeed, suppose that $(u_n)\subset (\widetilde{\Lambda}(c,\cdot))^\lambda$ is a Cauchy sequence with respect to the Finsler metric of $S_{\mathcal{C}}$. Then, as a sequence in $X\setminus\{0\}$ we have two possibilities: $u_n\to u\in \mathcal{S}_\mathcal{C}$ or $u_n\to  u\in \partial\mathcal{C}$. Thanks to condition $(H2)$ the second possibility is ruled out, so $u_n\to u\in \mathcal{S}_\mathcal{C}$ and the claim is proved. By \cite[Theorem 3.1]{Sz} we conclude that $\lambda_{c,k}$ is a critical value of $\widetilde{\Lambda}(c,\cdot)$. Now we can use \eqref{equi} and \eqref{derivatives} to obtain  $u_{c,k}\in \mathcal{C}$ such that 
	\begin{equation*}
		\Phi_{\lambda_{c,k}}(u_{c,k})=c\ \ \mbox{and}\ \ 	\Phi_{\lambda_{c,k}}'(u_{c,k})=0.
	\end{equation*}

To conclude, suppose that $\gamma(\mathcal{S}_\mathcal{C})=\infty$. Then $\lambda_{c,k}<\infty$ for all $k\in \mathbb{N}$ and we can assume that $\lim_{k\to \infty}\lambda_{c,k}= \lambda\in (0,\infty]$. We claim that $\lambda=\infty$. If not, then we can apply Lemma \ref{genus} to conclude that $\gamma(K_\lambda)=\infty$, which is a contradiction.

\end{proof}

\subsection{Conditions $(C1)$-$(C3)$ imply $(H1)$ and $(H2)$ } \label{abstractarguments}
In this section, we assume that $\Phi_\lambda$ is given by \eqref{dephi}, 
 under conditions $(C1)$-$(C3)$. Our goal is to show that conditions $(H1)$ and $(H2)$ are satisfied.
 Clearly for any $u\in X$ such that $A(u)\neq 0$ and $c\in \mathbb{R}$ we have
\begin{equation*}
	\lambda(c,u)=\frac{\frac{1}{\eta}N(u)-\frac{1}{\beta}B(u)-c}{\frac{1}{\alpha}A(u)}.
\end{equation*}

Let us show the existence of $I$ and $\mathcal{C}$ for which conditions $(H1)$ and $(H2)$ hold true. To this end we study the fibering maps associated with $\lambda(c,\cdot)$: given $u\in \mathcal{C}_A$, we write 
\begin{equation*}
	\varphi_{c,u}(t):=\lambda_c(tu)=\frac{\frac{t^{\eta-\alpha}}{\eta}N(u)-\frac{t^{\beta-\alpha}}{\beta}B(u)-t^{-\alpha}c}{\frac{1}{\alpha}A(u)}, \quad \forall t>0.
\end{equation*}
Then it is clear that
\begin{equation}\label{Nehari}
	\mathcal{N}_c=\{u\in \mathcal{C}_A:\ \varphi'_{c,u}(1)=0\}=\left\{u\in \mathcal{C}_A:\ \frac{\eta-\alpha}{\eta}N(u)-\frac{\beta-\alpha}{\beta}B(u)+\alpha c=0\right\}.
\end{equation}
We write
\begin{equation*}
	\mathcal{N}_c^+=\{u\in 	\mathcal{N}_c:\ \varphi''_{c,u}(1)>0\},
\quad \mbox{and} \quad
	\mathcal{N}_c^-=\{u\in 	\mathcal{N}_c:\ \varphi''_{c,u}(1)<0\}.
\end{equation*}

Let us prove that for suitable values of $c$ the Nehari sets $\mathcal{N}_c^\pm$ are non-empty, symmetric, $C^1$-Finsler manifolds and also natural constraints to $\lambda_c$. Recall that
\begin{equation*}
	 \mathcal{C}_A\ \mbox{and}\ \mathcal{C}_B\ \mbox{are open cones}.
\end{equation*}
We start with a technical lemma whose proof is straightforward.

\begin{lemma} \label{extremalnehari} Suppose $(C1)$. Then for any $u\in  \mathcal{C}_A$ the system $\varphi_{c,u}'(t)=\varphi''_{c,u}(t)=0$ has a solution $(c,t)\in(\mathbb{R},(0,\infty))$ if, and only if, $u\in  \mathcal{C}_B$. Moreover, in this case the solution is unique, and given by
	\begin{equation*}
	t(u) = \left[ \frac{(\eta - \alpha) N(u)}{(\beta - \alpha) B(u)} \right]^{\frac{1}{\beta-\eta}},
\end{equation*}
and
\begin{equation}\label{defc}
	c(u) = -\frac{(\eta-\alpha)(\beta-\eta)}{\eta\beta\alpha}\left(\frac{\eta-\alpha}{\beta-\alpha}\right)^{\frac{\eta}{\beta-\eta}}\frac{N(u)^{\frac{\beta}{\beta-\eta}}}{B(u)^{\frac{\eta}{\beta-\eta}}}.
\end{equation}
\end{lemma}

\begin{lemma}\label{extremalbounded} Suppose $(C1)$ and $(C2)$. Then the functional $c(u)$, defined by \eqref{defc} for $u\in C_A\cap\mathcal{C}_B$, is bounded away from zero.
\end{lemma}
\begin{proof} From \eqref{defc} it is clear that the functional $c(u)$ is $0$-homogeneous. Therefore it suffices to prove that 
	\begin{equation*}
		\sup_{u\in S\cap \mathcal{C}_A\cap\mathcal{C}_B}c(u)<0.
	\end{equation*}
From $(C2)$ we know that $B$ is bounded from above and $N$ is away from zero in $S\cap \mathcal{C}_A\cap\mathcal{C}_B$,
so the desired conclusion follows from the expression of $c(u)$.
\end{proof}
By Lemma \ref{extremalbounded} we have that
\begin{equation*}
	c^*:=\sup_{u\in \mathcal{C}_A\cap\mathcal{C}_B}c(u)<0.
\end{equation*}
\begin{lemma}\label{neharisets}  Suppose $(C1)$ and $(C2)$.
	\begin{enumerate}
		\item[i)] Let $u\in\mathcal{C}_A\setminus \mathcal{C}_B$.
		\begin{enumerate}
			\item If $c\ge 0$ then $\varphi_{c,u}$ has no critical points.
			\item If $c<0$ then $\varphi_{c,u}$ has a unique  nontrivial critical point $t_c^+(u)$, which is  a global minimizer of Morse type. 
		\end{enumerate}
\item[ii)] Let $u\in \mathcal{C}_A\cap\mathcal{C}_B$.
\begin{enumerate}
				\item If $c\ge 0$ then $\varphi_{c,u}$ has a unique  nontrivial critical point $t_c^-(u)$, which is  a global maximizer of Morse type.  
				\item If $c\in(c^*,0)$ then $\varphi_{c,u}$ has exactly two nontrivial critical points  $t_c^+(u)<t_c^-(u)$, which are, respectively, a local minimizer and a local maximizer, both of Morse type. 
			\end{enumerate}
\end{enumerate}
\end{lemma}

\begin{proof} We prove only $ii) (b)$. Indeed, by the definition of $c^*$ we know that $c(u)\le c^*$ for all $u\in \mathcal{C}_A\cap\mathcal{C}_B$, which implies, by Lemma \ref{extremalnehari}, that the system $\varphi_{c,u}'(t)=\varphi''_{c,u}(t)=0$ has no solution for $c>c^*$. Now one can easily see that for any $c$ the equation $\varphi_{c,u}'(t)=0$ has at most two solutions $t_c^+(u)<t_c^-(u)$ and this happens if $c^*<c<0$. Since  $\varphi''_{c,u}$ does not vanish at $t_c^+(u)$ and $t_c^-(u)$, the proof is complete.
\end{proof}

Write $I^+=(c^*,0)$ and $I^-=(c^*,\infty)$.

\begin{proposition}\label{neharimanifolds}  Suppose $(C1)$ and $(C2)$. Then:
	\begin{enumerate}
		\item[i)] Condition $(H1)$ holds true with $I=I^-$ and $\mathcal{C}=\mathcal{C}_A\cap\mathcal{C}_B$. Moreover
		\begin{equation*}
			\mathcal{N}_c^-=\{t^-(c,u)u:\ u\in \mathcal{C}_A\cap\mathcal{C}_B\}.
		\end{equation*}
\item[ii)] Condition $(H1)$ holds true with $I=I^+$ and $\mathcal{C}=\mathcal{C}_A$. Moreover
\begin{equation*}
				\mathcal{N}_c^+=\{t^+(c,u)u:\ u\in \mathcal{C}_A\}.
			\end{equation*}
\end{enumerate}
\end{proposition}

\begin{proof} The proof is a straightforward consequence of Lemma \ref{neharisets}.
\end{proof}

In the sequel we omit the symbols $+$ or $-$ when there is no need to differentiate both cases. Note by \eqref{Nehari} that
\begin{equation}\label{funcionalnehari}
	\Lambda(c,u)=\frac{\frac{\beta-\eta}{\eta}N(u)- \beta c}{\frac{\beta-\alpha}{\alpha} A(u)},\quad \mbox{for} \quad u\in \mathcal{N}_c.
\end{equation}
Now we will study in which circumstances condition $(H2)$ is verified. To this end, we first study the boundary of $\mathcal{N}_c^\pm$.

\begin{lemma}\label{boundary} Suppose $(C1)$ and $(C2)$. Then $\overline{\mathcal{N}}_c\subset \overline{C_A}\setminus \{0\}$ for all  $c\in I$. In particular $\mathcal{N}_c$ is away from $0$.

\end{lemma}
\begin{proof} By Proposition \ref{neharimanifolds} we have that $\overline{\mathcal{N}_c^+}\subset \overline{C_A}$ and $\overline{\mathcal{N}_c^-}\subset \overline{C_A\cap C_B} \subset \overline{C_A}$, so it is remains to show that $\mathcal{N}_c$ is away from zero. If $c\neq 0$, then this is clear from \eqref{Nehari}, while if $c=0$ we can use inequality (5.5) in  \cite[Lemma 5.4]{MR4736027} which clearly is true in our case assuming that $B(u)>0$ (in fact it does holds for all $c>c^*$), that is,
	\begin{equation}\label{N-bounded}
	1> \left(\frac{\eta-\alpha}{\beta-\alpha}\frac{N(u)}{B(u)}\right)^{\frac{1}{\beta-\eta}}\ge \frac{C}{\|u\|},\ \forall u\in \mathcal{N}_c^-,\ c>c^*,
	\end{equation}
where, in the second inequality we have used $(C2)$. 
\end{proof}

\begin{remark}\label{rmkcomplete} Under conditions $(C1)$, $(C2)$ and $(C4)$, the functional $A$ is away from zero on any bounded subset of $\mathcal{N}_c^-$. Indeed, if there exists a bounded sequence $(u_n) \subset \mathcal{N}_c^-$ such that $A(u_n) \to 0$ then $(C4)$ yields that $B(u_n) \to 0$. However, this is in contradiction with \eqref{N-bounded}, which proves the claim.
	As a consequence, the manifold $\mathcal{N}_c^-$ is complete for any $c>c^*$. Indeed, we already know from Lemma \ref{boundary} that $\overline{\mathcal{N}_c^-}\subset \overline{C_A}\setminus \{0\}$. If $\overline{\mathcal{N}_c^-} \neq \mathcal{N}_c^-$ then there exists a sequence $(u_n) \subset \mathcal{N}_c^-$ such that $u_n \to u \not \in \mathcal{N}_c^-$. Thus $(u_n) \subset \mathcal{C}_A$ and $\phi_{c,u_n}'(1)=0>\phi_{c,u_n}''(1)$ for every $n$. By continuity we deduce that $\phi_{c,u}'(1)=0 \ge \phi_{c,u}''(1)$, and since $u \not \in \mathcal{N}_c^-$ we must have either $u \not \in \mathcal{C}_A$ or $\phi_{c,u}''(1)=0$. Since  $A$ is away from zero on any bounded subset of $\mathcal{N}_c^-$ the first possibility is ruled out, so $\phi_{c,u}''(1)=0$. However, this is impossible since $c>c^*$. Thus we reach a contradiction and we conclude that $\overline{\mathcal{N}}_c=\mathcal{N}_c$.
\end{remark}

Next we prove that $\Lambda(c,\cdot)$ is coercive, that is, if $(u_n)\subset \mathcal{N}_c$ approaches $\partial \mathcal{N}_c$ or is unbounded, then
$\Lambda(c,u_n)\to \infty$. Thanks to Lemma \ref{boundary} we need to understand the behavior of $A(u)$ near the boundary of $\mathcal{N}_c$.

\begin{lemma}\label{topolopro} Suppose $(C1)$ and $(C2)$. Then:
	\begin{enumerate}
		\item[i)] For any $c\in I^+$ the set $\mathcal{N}_c^+$ is bounded and the functional $u \mapsto \lambda(c,u)$ is positive and bounded away from zero on $\mathcal{N}_c^+$. Furthermore  $\Lambda^+(c,u)\to \infty$ if  $A(u) \to 0$.
		\item[ii)] There exists $c^{**}>0$ such that for any $c\in (c^*,c^{**})$
		 the functional $u \mapsto \lambda(c,u)$ is positive, bounded away from zero, and coercive on $\mathcal{N}_c^-$, that is, $\Lambda^-(c,u)\to \infty$ if $u \in \mathcal{N}_c^-$ and either $\|u\|\to \infty$ or $A(u)\to 0$. 
		\item[iii)] If we assume, in addition, $(C4)$ then $u \mapsto \lambda(c,u)$ is  coercive on $\mathcal{N}_c^-$ for any $c>c^*$.
	\end{enumerate}
\end{lemma}

\begin{proof}  \strut
	\begin{enumerate}
\item[i)] 	We can proceed as in \cite[Lemma 5.4]{MR4736027} to show that
	\begin{equation}\label{ine+}
	t^+(c,u)<\left(-\frac{\alpha\beta\eta c}{(\beta-\eta)(\eta-\alpha)}\frac{1}{N(u)}\right)^{\frac{1}{\eta}},\ \forall c\in(c^*,0),\ u\in S_{\mathcal{C}_A},
	\end{equation}
	which yields the boundedness of $\mathcal{N}_c^+$. In addition, one can show that
	$\lambda(c,t^+(c,u)u)\ge -C_1 c t^+(c,u)^{-\alpha}$ for some $C_1>0$ and any $c\in(c^*,0)$ and $u\in S_{\mathcal{C}_A}$. Thus \eqref{ine+} implies that $\lambda(c,u) \ge C>0$ for any $c\in(c^*,0)$ and $u \in \mathcal{N}_c^+$. 
	Lastly, if $u \in \mathcal{N}_c^+$ and $A(u)\to 0$ then \eqref{funcionalnehari} yields $\Lambda(c,u)\to \infty$.
	\item [ii)] First of all we note that \eqref{funcionalnehari} and $(C2)$ yield 
	\begin{equation}\label{eqco}
	\Lambda^-(c,u) \ge C \frac{\frac{\beta-\eta}{\eta}\|u\|^{\eta}- \beta c}{\frac{\beta-\alpha}{\alpha} \|u\|^{\alpha}} \to \infty, \quad \mbox{if } u \in \mathcal{N}_c^- \mbox{ and } \|u\|\to \infty.
	\end{equation}
	Let us now show the existence of $c^{**}$.
	Fix $u\in \mathcal{C}_A\cap \mathcal{C}_B$ and consider the system $\varphi_{c,u}(t)=\varphi'_{c,u}(t)=0$. As in Lemma \ref{extremalnehari} one can show that this system has a unique solution $(t,c)\in((0,\infty),(0,\infty))$, given by
	\begin{equation*}
	t:=	t_0(u) =\left( \frac{N(u)}{B(u)} \right)^{\frac{1}{\beta - \eta}},
	\quad \mbox{and} \quad
	c:=	c_0(u) =	 \frac{\beta-\eta}{\eta\beta}\frac{N(u)^{\frac{\beta}{\beta - \eta}}}{B(u)^{\frac{\eta}{\beta-\eta}}}.
	\end{equation*}
	We set $c^{**}:=\displaystyle \inf_{u\in \mathcal{C}_A\cap\mathcal{C}_B} c_0(u)$ and observe by $(C2)$  that  $c^{**}>0$. We claim that 
	for any $c\in (c^*,c^{**})$ there exists a constant $C>0$ such that
	\begin{equation}\label{inec0} \Lambda^-(c,u)\ge \frac{C}{A(u)}>0, \quad \forall u \in \mathcal{N}_c^-,\end{equation}
	which combined with \eqref{eqco} yields the desired conclusion.
	Indeed, 
	 by Lemma \ref{boundary} we know that $\mathcal{N}_c^-$ is away from $0$, so for any $c\le 0$ there exists $C>0$ such that $$\frac{\beta-\eta}{\eta}N(u)- \beta c\geq C, \quad \forall u \in \mathcal{N}_c^-,$$
	 and \eqref{funcionalnehari} implies \eqref{inec0}.
	   Let now $c\in(0,c^{**})$. By Lemma \ref{neharisets} we know that $	\Lambda^-(c,u)\ge \varphi_{c,u}(t) $ for any $t>0$. On the other hand, since $\varphi_{c_0(u),u}(t_0(u))=0$ yields $$c_0(u)=\frac{t_0(u)^{\eta}}{\eta}N(u)-\frac{t_0(u)^{\beta}}{\beta} B(u),$$ we infer that
$$
	\Lambda^-(c,u)\ge \varphi_{c,u}(t_0(u)) =\alpha \frac{c_0(u)-c}{A(t_0(u)u)} \ge \alpha \frac{c^{**}-c}{t_0(u)^{\alpha}A(u)}, \quad \forall u \in \mathcal{N}_c^-.
$$
	By \eqref{N-bounded} we know that $u \mapsto t_0(u)$ is bounded on $\mathcal{N}_c^-$, which yields \eqref{inec0}, 
	
	\item[iii)] If $(C4)$ holds then $A$ is away from zero in any bounded subset of $\mathcal{N}_c^-$ (see Remark \ref{rmkcomplete}), so \eqref{eqco} yields the conclusion.
\end{enumerate}
\end{proof}

\begin{remark}
Note that the values $t_0(u),c_0(u)$ in the proof of of Lemma \ref{topolopro} - ii) satisfy $t_0(u)=t^-(c_0(u),u)$ for any $u\in \mathcal{C}_A\cap\mathcal{C}_B$.\\
\end{remark}

Under conditions $(C1)$ and $(C2)$ we set
\begin{equation*}
	J^+:=I^+=(c^*,0)\ \mbox{and}\ J^-:=(c^*,c^{**}).
\end{equation*}
If we assume, in addition, $(C4)$, then we set $J^-:=(c^*,\infty)$.

The previous results contain, in particular, the next one:

\begin{lemma}\label{coerciveboundary} Suppose $(C1)$ and $(C2)$ or $(C1)$, $(C2)$ and $(C4)$. Take $c\in J$. If $(u_n)\subset \mathcal{N}_c$ satisfies either $A(u_n)\to 0$ or $\|u_n\|\to \infty$, then $\Lambda(c,u_n)\to \infty$.
\end{lemma}

We are now in position to verify condition $(H2)$:

\begin{proposition}\label{H2}  Suppose $(C1)$-$(C3)$ or $(C1)$-$(C4)$. Then
	\begin{enumerate}
		\item[i)]  Condition $(H2)$ holds true with $I=J^-$ and $\mathcal{C}=\mathcal{C}_A\cap\mathcal{C}_B$. 
		\item[ii)] Condition $(H2)$ holds true with $I=J^+$ and $\mathcal{C}=\mathcal{C}_A$. 
	\end{enumerate}
\end{proposition}
\begin{proof} By Lemma \ref{topolopro} we know that $\Lambda(c,u)$  is bounded from below on $\mathcal{N}_c$, so $\widetilde{\Lambda}(c,u)$ is bounded from below in $\mathcal{S}_\mathcal{C}$, i.e. $(H2)$-$(a)$ is proved.  Now suppose that $(u_n)\subset \mathcal{N}_c$ is a Palais-Smale sequence of $\Lambda(c,\cdot)$. Since
	\begin{equation*}
			\frac{\partial \lambda}{\partial u}(c,u_n)u_n=0, \forall n\in \mathbb{N},
	\end{equation*}
it follows that $(u_n)$ is a Palais--Smale sequence of $\lambda(c,\cdot)$. Note by \eqref{derivatives} that
\begin{equation*}
	\frac{\partial \lambda}{\partial u}(c,u_n)=\frac{\Phi'_{\lambda(c,u_n)}(u_n)}{A(u_n)}.
	\end{equation*}
From Lemma \ref{topolopro} we know that $(A(u_n))$ is away from $0$, and $(u_n)$ is bounded. It follows that $(\lambda(c,u_n))$ is bounded,so by condition $(C3)$ we can assume that $u_n\to u$ with $A(u)>0$. By Lemma \ref{boundary}, the functional $\Lambda(c,\cdot)$ satisfies the Palais--Smale condition, which implies $(H2)$-$(b)$. Finally, $(H2)$-$(c)$ follows from Lemma \ref{coerciveboundary}. Indeed, recall that $\widetilde{\Lambda}(c,u)=\Lambda(c,t(c,u)u)$. Thus, if $A(u_n) \to 0$ then we have two possibilities: $t(c,u_n)\to \infty$, or $(t(c,u_n))$ is bounded. The first one implies that $\|t(c,u_n)u_n\|=t(c,u_n)\to \infty$, while the second one implies that $A(t(c,u_n)u_n)\to 0$, so by Lemma \ref{coerciveboundary} we conclude in both cases that $\widetilde{\Lambda}(c,u_n)\to \infty$.
\end{proof}

Summing up, from Propositions \ref{neharimanifolds} and \ref{H2} we obtain the following result:
\begin{proposition}\label{H12} Suppose $(C1)$-$(C3)$ or $(C1)$-$(C4)$. Then $(H1)$ and $(H2)$ hold true with $I=J^+$ and $\mathcal{C}=\mathcal{C}_A$ or $I=J^-$ and $\mathcal{C}=\mathcal{C}_A\cap\mathcal{C}_B$.
\end{proposition}

\subsection{Behavior of the energy curves} \label{curvesbeha}
All over this subsection we assume conditions $(C1)$-$(C3)$. We set
\begin{equation*}
	\lambda_{c,k}^+=\inf_{M \in \F_{k}}\, \sup_{u \in M}\, \widetilde{\Lambda}^+(c,u),\quad \mbox{for} \quad c\in J^+, \quad \mbox{and} \quad 1 \le k \le \gamma(S_{\mathcal{C}_A}),
\end{equation*}
 and
\begin{equation*}
	\lambda_{c,k}^-=\inf_{M \in \F_{k}}\, \sup_{u \in M}\, \widetilde{\Lambda}^-(c,u),,\quad \mbox{for} \quad \ c\in J^-, \quad \mbox{and} \quad 1 \le k \le \gamma(S_{\mathcal{C}_A \cap \mathcal{C}_B}).
\end{equation*}
Recall that $$\F_{k} = \bgset{M \in \F : M\ \mbox{is compact and}\  \gamma(M) \ge k},$$
where the $\F$ is the class of closed and symmetric subsets of $\mathcal{S}_\mathcal{C}=\mathcal{S} \cap \mathcal{C}$, with
$\mathcal{C}:=\mathcal{C}^+:=\mathcal{C}_A$ in the first case, and
$\mathcal{C}:=\mathcal{C}^-:=\mathcal{C}_A\cap \mathcal{C}_B$ in the second case.

In this section, we will study the curves $C_k=\{(\lambda_{c,k},c):\ c\in J\}$. We recall that the symbols $+$, $-$, will be dropped when there is no need to differentiate both cases. The ideas behind the proofs here comes from \cite{MR4736027}. For each $\lambda>0$ we denote $L_\lambda=\{(\lambda,c):\ c\in \mathbb{R}\}$. Our goal is to prove the following results:
\begin{theorem}\label{thmcurves+} Let $1 \le k \le \gamma(S_{\mathcal{C}_A})$. Then:
	\begin{enumerate}
		\item[i)] The map $c \mapsto \lambda_{c,k}^+$, is continuous and decreasing in $J^+$.
		\item[ii)] $\displaystyle \lim_{c\to 0^-}\lambda_{c,k}^+=0$.
		\item [iii)] If $\gamma(S_{\mathcal{C}_A})=\infty$, then for all $\lambda>0$ there exist two sequences $(k_n)\subset \mathbb{N}$ and $(c_n)\subset (c^*,0)$ such that $k_n\to \infty$, $c_n\to 0^-$ and $\lambda=\lambda_{c_n,k_n}^+$ for every $n$. Moreover, if $v_n\in S_{\mathcal{C}_A}$ satisfies $\widetilde{\Lambda}^+(c_n,v_n)=\lambda$, then $\|t^+(c_n,v_n)v_n\|=t^+(c_n,v_n)\to 0$.
	\end{enumerate}
\end{theorem}

\begin{theorem}\label{thmcurves-} Let $1 \le k \le \gamma (S_{\mathcal{C}_A \cap \mathcal{C}_B})$. Then:
	\begin{enumerate}
		\item[i)] The map $c \mapsto \lambda_{c,k}^-$ is continuous and decreasing in $J^-$. Moreover $\lambda_{c,k}^+<\lambda_{c,k}^-$ for any $c\in J^+$.
		\item[ii)]  If $(C4)$ holds then  $\displaystyle \lim_{c\to \infty}\lambda_{c,k}^-=-\infty$.
		\item[iii)] If $(C4)$ holds and $\gamma(S_{\mathcal{C}_A\cap \mathcal{C}_B})=\infty$, then for all $\lambda\in \mathbb{R}$ there exist two sequences $(k_n)\subset \mathbb{N}$ and $(c_n)\subset (c^*,\infty)$ such that $k_n\to \infty$, $c_n\to \infty$ and $\lambda=\lambda_{c_n,k_n}^-$ for every $n$. Moreover, if $u_n\in S_{\mathcal{C}_A\cap \mathcal{C}_B}$ satisfies $\widetilde{\Lambda}^-(c_n,u_n)=\lambda$, then $\|t^-(c_n,u_n)u_n\|=t^-(c_n,u_n)\to \infty$.
	\end{enumerate}

\end{theorem}
We prove the above theorems relying on the next results:

\begin{lemma} \label{Lemma 1}
	The following assertions hold:
	\begin{enumroman}
		\item[i)] For any $u \in S_{\mathcal{\mathcal{C}}}$, the map $ c \mapsto \widetilde{\Lambda}(c,u)$ is decreasing in $J$.
		\item[ii)] For any $1 \le k \le \gamma(S_{\mathcal{C}})$, the map $c \mapsto \lambda_{c,k}$ is nonincreasing in $J$.
	\end{enumroman}
\end{lemma}

\begin{proof} $i)$ is a straightforward application of the Implicit Function Theorem (see \cite[Section 3]{MR4736027} where $X\setminus\{0\}$ has to be replaced by $\mathcal{C}$). For further use let us register here that 
	\begin{equation}\label{inedecrea}
		 \frac{\partial \widetilde{\Lambda}(c,u)}{\partial c}=-\frac{1}{A(t(c,u)u)},\ u\in \mathcal{S}_\mathcal{C}.
	\end{equation}
To prove $ii)$, given $c_1<c_2$, note by $i)$ that
\begin{equation*}
		\lambda_{c_2,k}=\inf_{M \in \F_{k}}\, \sup_{u \in M}\, \widetilde{\Lambda}(c_2,u)\le \inf_{M \in \F_{k}}\, \sup_{u \in M}\, \widetilde{\Lambda}(c_1,u) \le \lambda_{c_1,k}.
		\end{equation*}
\end{proof}

Lemma \ref{Lemma 1} implies that we can work in a suitable sublevel set $\widetilde{\Lambda}(c,\cdot)^T = \bgset{u \in S_{\mathcal{\mathcal{C}}} : \widetilde{\Lambda}(c,u) \le T}$, as shown by the next result:

\begin{lemma} \label{Lemma 2}
	Let $[a,b] \subset J$, $1 \le k \le \gamma(\mathcal{S}_\mathcal{C})$, and $T > \lambda_{a,k}$. Denote by $\F_{b,T}$ the class of symmetric, closed subsets of $\widetilde{\Lambda}(b,\cdot)^T$ and let $\F_{b,T,k} = \bgset{M \in \F_{b,T}:\ $M$\ \mbox{is compact and}\ \gamma(M) \ge k}$. Then
	\[
	\lambda_{c,k}=\inf_{M \in \F_{b,T,k}}\, \sup_{u \in M}\, \widetilde{\Lambda}(c,u)  \quad \forall c \in [a,b].
	\]
\end{lemma}

\begin{proof}
	Clearly $\F_{b,T,k} \subset \F_k$. In addition, if $M \in \F_k \setminus \F_{b,T,k}$ and $c \in [a,b]$ then, 	by Lemma \ref{Lemma 1}, we have
	\[
	\sup_{u \in M}\, \widetilde{\Lambda}(c,u) \ge \sup_{u \in M}\, \widetilde{\Lambda}(b,u) > T >	\lambda_{a,k} \ge \lambda_{c,k},
	\]
    which provides the desired conclusion.
\end{proof}

Next we show that the curves $C_k$ are continuous and decreasing:

\begin{proposition} \label{Proposition 1}
	For any $1\le k \le \gamma(\mathcal{S}_\mathcal{C})$, the map $ c \mapsto \lambda_{c,k}$ is continuous and decreasing in $J$.
\end{proposition}

\begin{proof}
	Let $[a,b] \subset J$ and $T > \lambda_{a,k}$. By Lemma \ref{Lemma 2} we know that
	\begin{equation} \label{39}
	\lambda_{c,k}=	\inf_{M \in \F_{b,T,k}}\, \sup_{u \in M}\, \widetilde{\Lambda}(c,u)  \quad \forall c \in [a,b].
	\end{equation}
Given $u \in \widetilde{\Lambda}(b,\cdot)^T$, by the mean value theorem we have that
\[
	\widetilde{\Lambda}(b,u) - \widetilde{\Lambda}(a,u) = \frac{\partial \widetilde{\Lambda}(c,u)}{\partial c}\, (b - a)
	\]
for some $c \in (a,b)$. We claim that $A(t(c,u)u)$ remains bounded and away from zero in $[a,b] \times \widetilde{\Lambda}(b,\cdot)^T$, which implies, by \eqref{inedecrea}, that 
\[
	- C\, (b - a) \le \widetilde{\Lambda}(b,u) - \widetilde{\Lambda}(a,u) \le - C^{-1}\, (b - a)
	\]
for any $u \in \widetilde{\Lambda}(b,\cdot)^T$ and some $C>0$. Indeed, if $A(t(c,u)u)$ is not away from zero in $[a,b] \times \widetilde{\Lambda}(b,\cdot)^T$ then there exists a sequence $((c_n,u_n)) \subset [a,b]\times \mathcal{S}_\mathcal{C}$ such that $A(u_n)\to 0$ (by Lemma \ref{boundary}, $t(c_n,u_n)$ is away from zero). Since $\widetilde{\Lambda}(b,u)$ satisfies $(H2)$, by Proposition \ref{H2}, we conclude that $\widetilde{\Lambda}(b,u)$ is unbounded in $\widetilde{\Lambda}(b,\cdot)^T$, a contradiction. Now, if $A(t(c,u)u)$ is unbounded in $[a,b] \times \widetilde{\Lambda}(b,\cdot)^T$, then there exists a sequence $((c_n,u_n)) \subset [a,b]\times \mathcal{S}_\mathcal{C}$ such that $t(c_n,u_n)\to \infty$. By \cite[Lemma 5.4]{MR4736027} we know that $t^+$ is bounded from above in $J^+ \times \mathcal{S}$, thus $t(c_n,u_n)=t^-(c_n,u_n)$, and by \eqref{Nehari} and $(C3)$ it follows that $B(u_n)\to 0$. Now, since $\Lambda^-(b,u_n) \leq T$ we know by Lemma \ref{topolopro} that $(t^-(b,u_n))$ is bounded. However, 
 by \eqref{N-bounded}, we have that $$1>t^-(b,u_n)^{-1}\left(\frac{\eta-\alpha}{\beta-\alpha}\frac{N(u_n)}{B(u_n)}\right)^{\frac{1}{\beta-\eta}},$$ which contradicts $B(u_n) \to 0$.

Thus the claim is proved, and \eqref{39} yields
\[
	- C\, (b - a) \le \lambda_{b,k} - \lambda_{a,k} \le - C^{-1}\, (b - a),
	\]
from which the desired conclusions follow.
\end{proof}
\begin{proposition}\label{Lemma 4} Let $c \in J$. Then:
	
	\begin{enumerate}
	\item[i)]  $\displaystyle \lim_{c\to 0^-}\lambda_{c,k}^+=0$ for any $1\le k\le \gamma(S_{\mathcal{C}_A})$.
	\item[ii)] Suppose that $(C4)$ holds. Then $\displaystyle \lim_{c\to \infty}\lambda_{c,k}^-=-\infty$ for any $1\le k\le \gamma(S_{\mathcal{C}_A\cap \mathcal{C}_B})$. 
	\end{enumerate}

\end{proposition}

\begin{proof} $i)$ Indeed, from inequality (5.6) in \cite[Lemma 5.4]{MR4736027} we have
	\begin{equation}\label{ine+}
	t^+(c,u)<\left(-\frac{\alpha\beta\eta c}{(\beta-\eta)(\eta-\alpha)}\frac{1}{N(u)}\right)^{\frac{1}{\eta}},\ \forall c\in(c^*,0),\ u\in S_{\mathcal{C}_A}.
	\end{equation}
Now fix $M\in \F_k$. Since $M$ is a compact set in $S_{\mathcal{C}_A}$ it follows from $(C3)$  that $N$ and $A$ are bounded away from zero, and from above in $M$. Therefore
\begin{equation*}
	\sup_{u\in M}t^+(c,u)\le C|c|^{\frac{1}{\eta}},\ \forall c\in(c^*,0),
\end{equation*}
where $C>0$ is a constant. Consequently $\displaystyle \lim_{c\to 0^-}t^+(c,u)=0$ uniformly in $M$. From \eqref{Nehari} we also have
\begin{equation*}
	\frac{\eta-\alpha}{\eta}t^+(c,u)^\eta N(u)-\frac{\beta-\alpha}{\beta}t^+(c,u)^\beta B(u)+\alpha c=0,\ \forall c\in (c^*,0),
\end{equation*}
which implies that
\begin{equation*}
	\frac{\eta-\alpha}{\eta}t^+(c,u)^{\eta-\alpha}N(u)-\frac{\beta-\alpha}{\beta}t^+(c,u)^{\beta-\alpha} B(u)+\alpha t^+(c,u)^{-\alpha}c=0,\ \forall c\in (c^*,0).
\end{equation*}
Therefore $\displaystyle \lim_{c\to 0^-} t^+(c,u)^{-\alpha}c=0$ uniformly in $M$. To conclude, note by \eqref{funcionalnehari} that
\begin{equation*}
	\widetilde{\Lambda}^+(c,u)=\frac{\frac{\beta-\eta}{\eta}t^+(c,u)^{\eta-\alpha}\N(u)- \beta  t^+(c,u)^{-\alpha} c}{\frac{\beta-\alpha}{\alpha}A(u)},\ \mbox{for } c\in (c^*,0),\  u\in S_{\mathcal{C}_A},
\end{equation*}
which implies that $\displaystyle \lim_{c\to 0^-}	\widetilde{\Lambda}^+(c,u)=0$ uniformly in $M$. Hence
\begin{equation*}
0\le \lim_{c\to 0^-}\lambda_{c,k}^+\le  \lim_{c\to 0^-}\sup_{u\in M}\widetilde{\Lambda}(c,u)=0,
\end{equation*}
and the proof is complete.

$ii)$ Let $M\in \F_k$. Since $M$ is a compact set in $S_{\mathcal{C}_A\cap \mathcal{C}_B}$ it follows from $(C2)$  that $N$, $A$ and $B$ are bounded away from zero, and from above in $M$. From the definition of $\widetilde{\Lambda}(c,u)$ we conclude that
\begin{equation}\label{eqz1}
	\widetilde{\Lambda}(c,u)=\varphi_{c,u}(t^-(c,u))\le\psi_c(t^-(c,u))\le \sup_{t>0}\psi_c(t),
\end{equation}
where
\begin{equation*}
	\psi_c(t)=C_1t^{\eta-\alpha}-C_2t^{\beta-\alpha}-C_3t^{-\alpha}c,
\end{equation*}
and $C_1,C_2,C_3>0$ are constants not depending on $c$. Clearly $\psi_c$ has a unique global maximizer  $t(c)>0$, for any $c>0$. We claim that $\lim_{c\to \infty}\psi_c(t(c))=-\infty$. Indeed, note that
\begin{equation}\label{eqz2}
	C_1(\eta-\alpha)t(c)^{\eta-\alpha}-C_2(\beta-\alpha)t(c)^{\beta-\alpha}+C_3\alpha t(c)^{-\alpha}c=0,\ \forall c>0.
\end{equation}
By solving this equation with respect to $t(c)^{-\alpha}c$ and plugging it into $\psi_c(t(c))$ we obtain
\begin{equation}\label{eqz3}
	\psi_c(t(c))=\frac{C_1\eta}{\alpha}t(c)^{\eta-\alpha}-\frac{C_2\beta}{\alpha}t(c)^{\beta-\alpha},\ \forall c>0.
\end{equation}
Now observe from \eqref{eqz2} that $t(c)\to \infty$ as $c\to \infty$, so since $\beta>\eta$, we conclude from \eqref{eqz3} that $\displaystyle \lim_{c\to \infty}\psi_c(t(c))=-\infty$ and hence, by \eqref{eqz1}, it follows that $\displaystyle \lim_{c\to \infty}	\widetilde{\Lambda}(c,u)=-\infty$ uniformly in $u\in M$. Therefore
\begin{equation*}
	 \lim_{c\to \infty}\lambda_{c,k}^-\le  \lim_{c\to \infty}\sup_{u\in M}\widetilde{\Lambda}^-(c,u)=-\infty.
\end{equation*}
\end{proof}

\begin{proposition} \label{pdm} There holds $\lambda_{c,k}^+<\lambda_{c,k}^-$ for every $1 \le k \le \gamma (S_{\mathcal{C}_A\cap \mathcal{C}_B})$ and $c\in (c^*,0)$. 
\end{proposition}
\begin{proof} The proof follows the same arguments in the proof of \cite[Proposition 5.9]{MR4736027} after some changes. First we need to replace $S$ by use $S_{\mathcal{C}_A\cap \mathcal{C}_B}$. The the estimates on $t_n$, $N(u_n)$ and $A(u_n)$ in the proof of \cite[Lemma 5.10]{MR4736027} follow  from $(C2)$, inequality \eqref{ine+} and condition $(H2)$. Moreover, there is no need to assume that $u_n$ weakly converges to some $u$. Thus we can prove \cite[Corollary 5.11]{MR4736027} and complete the proof.
\end{proof}

We are now in position to prove the main results of this section.

\begin{proof}[Proof of Theorem \ref{thmcurves+}] \strut
	\begin{enumerate}
		\item[$i)$] It follows from Proposition \ref{Proposition 1}.
		\item[$ii)$] It follows from Proposition \ref{Lemma 4}. 
		\item[$iii)$] Fix $\lambda>0$ and $\overline{c}_1\in (c^*,0)$. By Theorem \ref{thm1} we know that $\displaystyle \lim_{k\to \infty}\lambda_{\overline{c}_1,k}^+=\infty$,  so there exists $k_1$ such that $\lambda_{\overline{c}_1,k_1}^+>\lambda$ and thus, by items $i)$ and $ii)$, there exists $c_1\in (\overline{c}_1,0)$ such that $\lambda_{c_1,k_1}^+=\lambda$. Arguing by induction, given a sequence $(\overline{c}_n)\subset (c^*,0)$ such that $\overline{c}_n<\overline{c}_{n+1}$ and $\overline{c}_n\to 0^-$, we can find two sequences $(k_n)$ and $(c_n)$ such that $k_n\to \infty$, $c_n\to 0^-$, $\overline{c}_n<c_n$ and $\lambda=\lambda_{c_n,k_n}^+$ for all $n$. Now suppose that $v_n\in S_{\mathcal{C}_A}$ satisfies $\widetilde{\Lambda}^+(c_n,v_n)=\lambda$. By \eqref{ine+} and $(C2)$ it follows that $\|t^+(c_n,v_n)v_n\|=t^+(c_n,v_n)\to 0$.
	\end{enumerate}
\end{proof}

\begin{proof}[Proof of Theorem \ref{thmcurves-}] \strut
		\begin{enumerate}
			\item[$i)$] It follows from Propositions \ref{Proposition 1} and \ref{pdm}.
		\item[$ii)$] It follows from Proposition \ref{Lemma 4}.
		\item[$iii)$] Fix $\lambda>0$ and $\overline{c}_1\in (c^*,\infty)$.  By Theorem \ref{thm1} we have that $\displaystyle \lim_{k\to \infty}\lambda_{\overline{c}_1,k}^-=\infty$, so we can find $k_1$ such that $\lambda_{\overline{c}_1,k_1}^->\lambda$ and thus, by items $i)$ and $ii)$, there exists $c_1\in (\overline{c}_1,\infty)$ such that $\lambda_{c_1,k_1}^-=\lambda$. Arguing by induction, given a sequence $(\overline{c}_n)\subset (c^*,\infty)$ such that $\overline{c}_n<\overline{c}_{n+1}$ and $\overline{c}_n\to \infty$, we can find, for each $n$, two sequences $(k_n)$ and $(c_n)$ such that $k_n\to \infty$, $c_n\to \infty$, $\overline{c}_n<c_n$ and $\lambda=\lambda_{c_n,k_n}^-$ for all $n$. Now suppose that $v_n\in S_{\mathcal{C}_A\cap \mathcal{C}_B}$ satisfies $\widetilde{\Lambda}^-(c_n,v_n)=\lambda$. Note that 
		\begin{equation*}
		\lim_{n\to \infty}	\Phi_\lambda(t^-(c_n,v_n)v_n)=\lim_{n\to \infty}c_n=\infty, 
		\end{equation*}
and since, by $(C2)$, $(N(t^-(c_n,v_n)v_n)), (A(t^-(c_n,v_n)v_n)), (B(t^-(c_n,v_n)v_n))$ are bounded if $(t^-(c_n,v_n)v_n)$ is bounded, it follows that $\|t^-(c_n,v_n)v_n\|=t^-(c_n,v_n)\to \infty$.
\end{enumerate}
\end{proof}

\subsection{Proof of Theorems \ref{thm2}, \ref{thm3}, \ref{thm4} and \ref{thm5}}
\begin{proof}[Proof of Theorem \ref{thm2}:]
	We apply Theorem \ref{thm1} with two choices of $I$ and $\mathcal{C}$, namely \begin{itemize} \item $I=(c^*,0)$ and $\mathcal{C}=\mathcal{C}_{\mathcal{A}}$
		\item $I=(c^*,c^{**})$ and $\mathcal{C}=\mathcal{C}_{\mathcal{A}} \cap \mathcal{C}_{\mathcal{B}}$.
	\end{itemize}	
		 Proposition \ref{H12} enables us to apply Theorem \ref{thm1}, which yields the existence of $\lambda_{c,k}^+$ and $v_{c,k}$ for any $c^*<c<0$ and $1\le k\le \gamma(\mathcal{S}_{\mathcal{C}_{\mathcal{A}}})=\gamma(\mathcal{C}_{\mathcal{A}})$, and the existence of $\lambda_{c,k}^-$ and $u_{c,k}$ for any $c^*<c<c^{**}$ and $1\le k\le \gamma(\mathcal{C}_{\mathcal{A}}\cap \mathcal{C}_{\mathcal{B}})$. We also note from Lemma \ref{topolopro} that $\lambda_{c,k}^+>0$ and $\lambda_{c,k}^->0$ for the corresponding values of $c$ and $k$. Theorems \ref{thmcurves+} and \ref{thmcurves-} provide the remaining properties of $\lambda_{c,k}^+$, $\lambda_{c,k}^-$, $v_{c,k}$, and $u_{c,k}$.
\end{proof}

\begin{proof} [Proof of Theorem \ref{thm3}] 
	
Let us write $\Phi_\lambda=\Phi_{\lambda,A}$ to stress the dependence of $\Phi_\lambda$ on $A$. Then it suffices to note that $\Phi_{\lambda,A}=\Phi_{-\lambda,-A}$, and apply Theorem \ref{thm2} to the latter functional. This procedure yields the values $\lambda_{c,k}^+(-A)$ and $\lambda_{c,k}^-(-A)$ for appropriate values of $c$ and $k$.
Then $\lambda_{c,k}^-=-\lambda_{c,k}^+(-A)$ and $\lambda_{c,k}^+=-\lambda_{c,k}^-(-A)$ have the desired properties.

\end{proof}

\begin{proof}[Proof of Theorem \ref{thm4}] Since this theorem complements Theorem \ref{thm2}, it is enough to check  items $vi)$ and $viii)$, which follow from Theorem \ref{thmcurves-}.
\end{proof}

\begin{proof} [Proof of Theorem \ref{thm5}] 
We can proceed as in the proof of Theorem \ref{thm3}.
\end{proof}

\section{Applications}\label{applica}
In this section, we prove Theorems \ref{thmapp1} and \ref{thmapp2}.

\begin{proof}[Proof of Theorem \ref{thmapp1}] Indeed, conditions $(C1)$-$(C3)$ are clearly satisfied.  In addition, since $A^+\cap B^+$ is an open subset of $\Omega$, it follows that $\mathcal{C}_A\cap\mathcal{C}_B\cup \{0\}$ contains the infinite dimensional vector space $H_0^1(U)$, where $U$ is an open set contained in $\mathcal{A}^+\cap \mathcal{B}^+$ (we extend functions by zero outside $U$). Therefore $\gamma(\mathcal{C}_A\cap\mathcal{C}_B)=\infty$, and we can apply Theorem \ref{thm2} to obtain all assertions except ii)-(c). 
Now assume that $a\ge 0$, and $\mathcal{A}^0 \subset \mathcal{B}^0$. 
If $(u_n)\subset \mathcal{C}_A$ is a bounded sequence satisfying $A(u_n)\to 0$ then we can assume that $u_n \rightharpoonup u$ with $A(u)=0$. It follows that supp$(u) \subset \mathcal{A}^0 \subset \mathcal{B}^0$ i.e. $B(u)=0$, which shows that condition $(C4)$ holds. Thus we can apply Theorem \ref{thm4} to obtain the desired conclusions.

\end{proof}

\begin{proof}[Proof of Theorem \ref{thmapp1.1}] As in the previous proof, we have now $\gamma(\mathcal{C}_{-A}\cap\mathcal{C}_B)=\infty$, so we apply Theorem \ref{thm3} to get the conclusion.
	
\end{proof}

\begin{proof}[Proof of Theorem \ref{thmapp2}] Conditions $(C1)$ and $(C2)$ follows a before. To prove condition $(C3)$ we note that in \cite[Lemma 2.3]{AmBoPe}, after proving that the sequence is bounded, they conclude that $u_n\to u$, which is exactly what we need. To prove $(C4)$ we argue as in the proof of Theorem \ref{thmapp1}. So we can assume that $(u_n)$ does not converge to zero and $v_n \rightharpoonup v$. However now we argue as in the proof of \cite[Lemma 2.3]{AmBoPe} to conclude that $A(v)=0$, $\operatorname{supp}(v)\subset \{x:\ a(x)\le 0\}$ and $0=B(v)=\lim_{n\to \infty}B(v_n)$. Theorem \ref{thm4} yields the desired conclusions.
\end{proof}

\section{Further examples and discussions} \label{Sfurther}

In this section, we discuss further properties related to problem \eqref{quasi}. Note that in \cite{BrWu}, the authors found two positive solutions to \eqref{quasi} for small values of $\lambda > 0$: one with positive energy and the other with negative energy. By comparing their results with our Theorem \ref{thmapp1.1} (over the cone $\mathcal{C}_A$), we observe that our energy curve $\lambda_{c,1}^+$, for $c \in (c^*, 0)$, corresponds exactly to the positive solutions with negative energy. However, our second energy curve $\lambda_{c,1}^-$, for $c \in (0, c^{**})$, does not correspond exactly to their result, since we do not know, a priori, the value of the limit $\lim_{c \to c^{**}} \lambda_{c,1}^-$. If this limit is zero, then it is clear that this curve yields a positive solution with positive energy for $\lambda>0$ small, and in that case, we would recover the results of \cite{BrWu}.

It is important to note that our method searches for solutions within the cone $\mathcal{C}_A$ and, by its very definition, any solution to problem \eqref{quasi} obtained through our approach will always satisfy $A(u) > 0$. This is not the case in \cite{BrWu}, so it is possible that the solutions with positive energy described there may lie outside the cone $\mathcal{C}_A$ for $\lambda > 0$ and small. In fact, $\lambda$ do not need to be small since the interval identified in \cite{BrWu} for the existence of positive solutions with positive energy can, in fact, be extended to a maximal interval, say $(0, \lambda^*)$, such that there exists $\lambda_0^* \in (0, \lambda^*)$ with the property that the energy is positive for all $\lambda < \lambda_0^*$, vanishes at $\lambda = \lambda_0^*$, and becomes negative for $\lambda > \lambda_0^*$ (see, for example, \cite{Il}). Since solutions with negative energy exist only within the cone $\mathcal{C}_A$, it follows that for larger values of $\lambda$, the sign of $A$ is positive.

Now we show a case where the curve $\lambda_{c,1}^-$, for $c \in (0, c^{**})$, of equation \eqref{quasi} can be extended further to some $c^{***}>c^{**}$ such that $\lambda_{c^{***},1}^-<0$. Indeed, let us go back to the functional $c_0$ introduced in the proof of Lemma \ref{topolopro} - ii), namely
$$c_0(u) :=	 \frac{\beta-\eta}{\eta\beta}\frac{N(u)^{\frac{\beta}{\beta - \eta}}}{B(u)^{\frac{\eta}{\beta-\eta}}} \quad \text{for } u\in \mathcal{C}_B.$$
We assume the following additional condition:

\begin{enumerate}
	\item[(C5)] $N$ is weakly lower semicontinuous and $B$ is weakly continuous.
\end{enumerate}
\begin{lemma}\label{extreatta} Under conditions $(C1)$-$(C5)$, we have that $\displaystyle \inf _{u\in \mathcal{C}_B} c_0(u)$ is achieved. 
\end{lemma}

\begin{proof} Indeed, it is clear that $c_0$ is a $C^1$ and $0$-homogeneous functional defined over the open cone $\mathcal{C}_B$. Therefore
	\begin{equation*}
	\inf _{u\in \mathcal{C}_B} c_0(u)=\inf _{u\in S_{\mathcal{C}_B}} c_0(u).
	\end{equation*}
	If $(u_n)\subset S_{\mathcal{C}_B}$ is a minimizing sequence, then we can assume that $u_n \rightharpoonup u$, $B(u_n) \to B(u)$ and clearly $u\neq 0$ since, on the contrary, by $(C3)$, $(C5)$ and the expression of $c_0$, we would conclude that $c_0(u_n)\to \infty$, a contradiction. Now observe that
	\begin{equation*}
	c_0(u/\|u\|)=	c_0(u)\le \liminf_{n\to \infty}c_0(u_n)=	\inf _{u\in \mathcal{C}_B} c_0(u),
	\end{equation*}
	and the proof is complete.
\end{proof}
Denote
\begin{equation*}
	M=\{w\in  S_{\mathcal{C}_B}: c_0(w)=\inf _{u\in S_{\mathcal{C}_B}} c_0(u)\}.
\end{equation*}
 
\begin{lemma}\label{Mcompact} Suppose $(C1)$-$(C5)$. Then $M$ is nonempty and sequentially weakly compact.
\end{lemma}
\begin{proof} The fact that $M\neq \emptyset$ follows from Lemma \ref{extreatta}. Now suppose that $(u_n)\subset M$. Then $(u_n)$ is a minimizing sequence for $\displaystyle \inf _{u\in \mathcal{C}_B} c_0(u)$ and we can proceed as in the proof of Lemma \ref{extreatta} to show that $u_n \rightharpoonup u\in M$. It is clear from the proof of Lemma \ref{extreatta} that $M$ is sequentially weakly compact.
\end{proof}

\begin{proposition}\label{c**attained} Assume that $M\subset \mathcal{C}_A$. Then $u\in S_{\mathcal{C}_A\cap\mathcal{C}_B}$ achieves $\lambda_{c^{**},1}^-$  if, and only if, $u\in M$. Moreover $\lambda_{c^{**},1}^-=0$.
\end{proposition}
\begin{proof} Note that 
	$$
	\Lambda^-(c^{**},u)\ge \varphi_{c^{**},u}(t_0(u)) =\alpha \frac{c_0(u)-c^{**}}{A(t_0(u)u)} \ge 0 \quad \forall u \in \mathcal{C}_A\cap\mathcal{C}_B.
	$$
	Hence $\widetilde{\Lambda}^-(c^{**},u)\ge 0$ for all $u\in  S_{\mathcal{C}_A\cap\mathcal{C}_B}$ and if $\widetilde{\Lambda}^-(c^{**},u)= 0$ then $c_0(u)=c^{**}$. Moreover, if $c_0(u)=c^{**}$ then 
	$$\Lambda^-(c^{**},u)= \varphi_{c^{**},u}(t^-(c^{**},u))=\varphi_{c^{**},u}(t^-(c_0(u),u))=\varphi_{c_0(u),u}(t_0(u)) =0,$$
	which completes the proof.
\end{proof}

\begin{lemma}\label{infinterior} Under the conditions of Proposition \ref{c**attained}, there exist $\delta>0$ and $\varepsilon>0$ such that 
	\begin{equation*}
	\inf_{u\in S_{\mathcal{C}_A\cap\mathcal{C}_B}}\widetilde{\Lambda}^-(c,u)=\inf_{u\in S_{\mathcal{C}_A\cap\mathcal{C}_B}\setminus A_\varepsilon}\widetilde{\Lambda}^-(c,u), \ \forall c\in [c^{**},c^{**}+\delta).
\end{equation*}
	where $A_\varepsilon=\{u\in \mathcal{C}_A: \int_\Omega a(x)|u|^q\le \varepsilon\}$. Moreover $M \subset S_{\mathcal{C}_A\cap\mathcal{C}_B}\setminus A_\varepsilon $.
\end{lemma}
\begin{proof} From Lemma \ref{Mcompact} it is clear that there exists $\varepsilon>0$ such that
	\begin{equation*}
		c^{**}=\inf_{u\in S_{\mathcal{C}_A\cap\mathcal{C}_B}\setminus A_\varepsilon}c_0(u),
	\end{equation*}
	or, equivalently,
	\begin{equation*}
		\inf_{u\in S_{\mathcal{C}_A\cap\mathcal{C}_B}\cap A_\varepsilon}c_0(u)>c^{**}+\delta_1,
	\end{equation*}
	where $\delta_1>0$ depends on $\varepsilon$. Arguing as in the proof of \eqref{inec0} we conclude that
	\begin{equation*}
		\Lambda^-(c,u)\ge \alpha\frac{c_0(u)-c}{A(t_0(u)u)}>\alpha\frac{c^{**}+\delta_1-c}{A(t_0(u)u)},\ u\in S_{\mathcal{C}_A\cap\mathcal{C}_B}\cap  A_\varepsilon.
	\end{equation*}
Therefore
	\begin{equation}\label{kkkinq}
	\inf_{u\in S_{\mathcal{C}_A\cap\mathcal{C}_B}\cap  A_\varepsilon} \Lambda^-(c,u)\ge 0,\ c\in (c^{**}-\delta,c^{**}+\delta),
\end{equation}
where $\delta\in (0,\delta_1)$. Now note, by definition of $c^{**}$, that 
\begin{equation*}
		\inf_{u\in S_{\mathcal{C}_A\cap\mathcal{C}_B}}\widetilde{\Lambda}^-(c,u)< 0,\ \forall c> c^{**},
\end{equation*}
which, combined with \eqref{kkkinq}, implies that
	\begin{equation*}
	\inf_{u\in S_{\mathcal{C}_A\cap\mathcal{C}_B}}\widetilde{\Lambda}^-(c,u)=\inf_{u\in S_{\mathcal{C}_A\cap\mathcal{C}_B}\setminus A_\varepsilon}\widetilde{\Lambda}^-(c,u), \ \forall c\in (c^{**},c^{**}+\delta).
\end{equation*}
The case $c=c^{**}$ follows from Proposition \ref{c**attained} and the fact that $M \subset S_{\mathcal{C}_A\cap\mathcal{C}_B}\setminus A_\varepsilon $ (which is clear from the definition of $\varepsilon$). 

\end{proof}

\begin{proposition}\label{c***}  Under the conditions of Proposition \ref{c**attained} there exists $\delta>0$ such that the curve $\lambda_{c,1}^-$ is well defined for all $c\in[c^{**},c^{**}+\delta)$. Moreover, the infimum $\lambda_{c,1}^-$ is attained, and satisfies $\lambda_{c^{**},1}^-=0$, and $\lambda_{c,1}^-<0$ if $c\in(c^{**},c^{**}+\delta)$, and $\lim_{c\to (c^{**})^-}\lambda_{c,1}^-=\lambda_{c^{**},1}^-=0$.
\end{proposition}
\begin{proof} We could adapt the arguments of Section \ref{abstractarguments}, but instead, let us give a more straightforward argument. The case $c=c^{**}$ was treated in Proposition \ref{c**attained}. It is clear, by definition of $c_0$, that $\lambda_{c,1}^-<0$ if $c\in(c^{**},c^{**}+\delta)$. Suppose $u_n\in  S_{\mathcal{C}_A\cap\mathcal{C}_B}$ is a minimizing sequence to $\lambda_{c,1}^-$. By Lemma \ref{infinterior} we can assume that $\int_\Omega a(x)|u_n|^qdx >\varepsilon$ for all $n$. Therefore we can assume that $u_n \rightharpoonup u\neq 0$. Lemma \ref{boundary} implies that $t^-(c,u_n)$ is bounded away from zero. We also have, as in the proof of Lemma \ref{topolopro}, that $t^-(c,u_n)$ is bounded, so we can suppose that $t^-(c,u_n)u_n \rightharpoonup tu$, where $t>0$. Moreover, it follows from \eqref{N-bounded} that $\int_\Omega b(x)|u|^pdx>0$. Now observe, from condition $(C6)$  and Lemma \ref{neharisets}, that (here we write $t_n=t^-(c,u_n)$ for simplicity)
	
	\begin{eqnarray*}
	\Lambda^-(c,t^-(c,u)u)&\le & \liminf_{n\to \infty}	\Lambda^-(c,t^-(c,u)u_n) \\
	&\le &  \liminf_{n\to \infty}	\Lambda^-(c,t_nu_n) \\
	&=& \lambda_{c,1}^-,
	\end{eqnarray*}
	which implies that $	\Lambda^-(c,t^-(c,u)u)= \lambda_{c,1}^-$. 
	
	Now we prove that $\lim_{c\to (c^{**})^-}\lambda_{c,1}^-=0$. Indeed, fix $w\in M$ and note that
	\begin{equation*}
		\lim_{c\to (c^{**})^-}\widetilde{\Lambda}^-(c,w)=0.
	\end{equation*}
Since
	\begin{equation*}
	0<\lambda_{c,1}^-=	\inf_{u\in S_{\mathcal{C}_A\cap\mathcal{C}_B}}\widetilde{\Lambda}^-(c,u)\le \widetilde{\Lambda}^-(c,w), \ \forall c\in (c^{**}-\delta,c^{**}),
	\end{equation*} 
it follows that $\lim_{c\to (c^{**})^-}\lambda_{c,1}^-=0$ and the proof is complete.
\end{proof}

Finally we show that the condition $M\subset \mathcal{C}_A$ is satisfied for the problem \eqref{quasi} with some suitable $a\in L^\infty(\Omega)$. Recall that in Section \ref{applica} we proved that conditions $(C1)$-$(C4)$ are satisfied. It is also clear that $(C5)$ is satisfied.  

\begin{lemma}\label{c^**A>0} Assume that $b^+$ vanishes in an open ball $B \subset \Omega$. Then there exists $a\in L^\infty(\Omega)$ such that $M\subset \mathcal{C}_A$ and thus $	\inf _{u\in \mathcal{C}_B} c_0(u)=c^{**}$.
\end{lemma}
\begin{proof} It is clear that $\int_\Omega b^+(x)|u|^{\alpha}dx >0$ for all $u\in M$. We fix a non-negative and non-trivial $\theta\in C_0^\infty(B)$ and extend it by zero over $\Omega$. Given $\varepsilon\ge 0$, define $a_\varepsilon(x)=b^+(x)-\varepsilon\theta(x)$. We claim that there exists $\varepsilon>0$ such that $\int_\Omega a_\varepsilon(x)|u|^{\alpha}dx >0$ for all $u\in M$. On the contrary, we can find a sequence $\varepsilon_n\to 0$ and $u_n\in M$ such that $\int_\Omega a_{\varepsilon_n}(x)|u_n|^{\alpha}dx \le 0$ for all $n\in \mathbb{N}$. By Lemma \ref{Mcompact} we can assume  that $u_n \rightharpoonup u$, so that $\int_\Omega b^+(x)|u|^{\alpha}dx \le 0$. However, this is a contradiction, since by Lemma \ref{Mcompact} we must have that $u\in M$ and $\int_\Omega b^+(x)|u|^{\alpha}dx > 0$. Therefore we can find $\varepsilon>0$ satisfying the claim and the desired conclusion holds with $a:=a_\varepsilon$.
\end{proof}

Now we can prove Theorem \ref{thmapp1.1conje}:

\begin{proof}[Proof of Theorem \ref{thmapp1.1conje}] The existence of $c^{***}$ follows from Proposition \ref{c***}. The continuous and decreasing behavior of $c \mapsto \lambda_{c,1}^-$ can be proved in the same was as in Section \ref{curvesbeha}, or one can argue directly from the definitions.
\end{proof}

As observed in Remark \ref{rmkcomplete} we believe that $\lambda_{c,1}^-$ (blue curve) can be joined to $\lambda_{c,1}^+$ (red curve), see Figure \ref{fig:combinedCC}. This claim is also supported by \cite[Theorem 1.1]{KaQuUm}.

\appendix

\section{A deformation lemma}

In this appendix we prove a deformation lemma that will be used in this work. In fact, it is a straightforward consequence of the results of \cite{CoDeMa}, however, we will write the details here for the reader's convenience. In this section, we use the same notation of \cite{CoDeMa}. Let $(X,d)$ be a metric space and $f:X\to \mathbb{R}$ a continuous function. We will need the following condition

\begin{enumerate}
	\item[(CS)] If $u_n\in X$ is a not convergent Cauchy sequence, then $f(u_n)\to \infty$. 
\end{enumerate}

\begin{theorem}[Deformation Lemma] \label{dl} Suppose $(CS)$. Fix $c\in \mathbb{R}$ and assume $f$ satisfies the Palais--Smale condition at level $c$. Then, given $\overline{\varepsilon}>0$, $\mathcal{O}$ a neighborhood of $K_c$ (if $K_c=\emptyset$, we allow $\mathcal{O}=\emptyset$) and $\lambda>0$, there exist $\varepsilon>0$ and $\eta:X\times[0,1]\to X$ continuous with:
	\begin{enumerate}
		\item[i)] $d(\eta(u,t),u)\le \lambda t$; 
		\item[ii)] $f(\eta(u,t))\le f(u)$; 
		\item[iii)] if $f(u)\notin (c-\overline{\varepsilon},c+\overline{\varepsilon})$, then $\eta(u,t)=u$; 
		\item[iv)] $\eta(f^{c+\varepsilon}\setminus\mathcal{O},1)\subset f^{c-\varepsilon}$.
	\end{enumerate}
\end{theorem}

The proof of Theorem \ref{dl} can be achieved in the same way as in the proof of \cite[Theorem 2.14]{CoDeMa}, as long as we prove

\begin{lemma} Suppose $(CS)$. Assume $C$ is a closed subset of $X$ and $\delta,\sigma>0$ such that
	\begin{equation*}
		\mbox{if}\ d(u,C)\le \delta,\ \mbox{then}\ |df|(u)>\sigma.
	\end{equation*}
Then there exists a continuous map $\eta:X\times [0,\delta]\to X$ such that
\begin{enumerate}
		\item[i)] $d(\eta(u,t),u)\le t$; 
		\item[ii)] $f(\eta(u,t))\le f(u)$; 
		\item[iii)] if $d(u,C)\ge \delta$, then $\eta(u,t)=u$; 
		\item[iv)] if $u\in C$, then $f(\eta(u,t))\le f(u)-\sigma t$.
	\end{enumerate}
\end{lemma}
\begin{proof} Indeed, we proceed as in the proof of \cite[Theorem 2.11]{CoDeMa} up to the point where completeness was needed. This happens in the claim that for all $(u,t)$ with $d(u,C)+t\le \delta$, there holds $\lim_{h}\tau_h(u)>t$. If the claim is not true they conclude that $\eta_h(u,\tau_h(u))$ is a Cauchy sequence in $\{v: d(v,c)\le \delta\}$. Now we prove that this sequence converges. In fact, if not, by condition $(CS)$ we know that $\lim_{h}f(\eta_h(u,\tau_h(u)))=\infty$, which contradicts inequality $f(\eta_h(u,\tau_h(u)))\le f(u)$ in \cite[Theorem 2.8]{CoDeMa}. Therefore the claim is true and the proof is complete.
\end{proof}

\newpage
{\bf Acknowledgement.}
This work was initiated while the third author held a post-doctoral position at the Florida Institute of Technology, Melbourne, United States of America, supported by CNPq/Brazil under Grant 201334/2024-0. It was later completed during the second author’s visit to Instituto de Matemática e Estatística of the Universidade Federal de Goiás, supported by FAPEG (Programa Pesquisador Visitante Estrangeiro 2025). The second and third authors are grateful to their host institutions for the warm hospitality extended during these periods. The third author is also partially supported by CNPq/Brazil under Grant 300849/2025-7.

\def\cprime{$''$}


\begin{thebibliography}{10}
	
		
	\bibitem{AmBrCe}
Antonio Ambrosetti, Haim Brezis and Giovanna Cerami.
	\newblock Combined effects of concave and convex nonlinearities in some elliptic problems.
	\newblock {\em Journal of Functional Analysis},	122: 519--543, 1994.

\bibitem{AGP} A. Ambrosetti; J. Garcia Azorero, I. Peral,  
\newblock Quasilinear equations with a multiple bifurcation. 
\newblock {\em Differential Integral Equations}, 10 (1997), no. 1, 37–-50.    
	
\bibitem{AmBoPe}
	Eleonora Amoroso, Gabriele Bonanno and Kanishka Perera.
	\newblock Nonlinear ellipitc $p$-Laplacian equations in the whole space.
	\newblock {\em Nonlinear Analysis},	236: 113--364, 2023.

\bibitem{BM} T. Bartsch, R. Mandel.
\newblock Infinitely many global continua bifurcating from a single solution of an elliptic problem with concave-convex nonlinearity.
\newblock {\em J. Math. Anal. Appl}, 433 (2016), no. 2, 1006-1036.    

\bibitem{BW}  T. Bartsch, M. Willem. \newblock On an elliptic equation with concave and convex nonlinearities.
\newblock {\em Proc. Amer. Math. Soc}, 123 (1995), no. 11, 3555-–3561.

\bibitem{BrWu}
Kenneth J. Brown and Tsung-Fang, Wu.
\newblock A fibering map approach to a semilinear elliptic boundary value problem.
\newblock {\em Electronic Journal of Differential Equations (EJDE)},
2007(2007):Paper No. 69, 9 p., 2007.

\bibitem{BrWu1}
Kenneth J. Brown and Tsung-Fang, Wu.
\newblock A fibering map approach to a potential operator equation and its applications.
\newblock {\em Differential Integral Equations},
22 (11/12): 1097--1114, 2009.

\bibitem{CoDeMa}
Jean-Noel Corvellec, Marco Degiovanni and Marco Marzocchi.
\newblock Deformation properties for continuous functionals and critica point theory.
\newblock {\em Topological Methods in Nonlinear Analysis}, 1:151--171, 1993.

\bibitem{FiGoUb}
Djairo G. de Figueiredo, Jean-Pierre Gossez and Pedro Ubilla.
\newblock  Local superlinearity and sublinearity for
indefinite semilinear elliptic problems.
\newblock {\em Journal of Functional Analysis}, 199: 452--467, 2003.
    
	

\bibitem{GP} J. Garcia Azorero and I. Peral Alonso, 
\newblock Multiplicity of solutions for elliptic problems with critical exponent or with a non-symmetric term.
\newblock {\em Trans. Amer. Math. Soc}, 323 (1991), 877--895.
\bibitem{GoSr}
Sarika Goyal and Konijeti Sreenadh.
\newblock Nehari manifold for non-local elliptic operator with concave–convex nonlinearities and sign-changing weight functions.
\newblock {\em Proc Math Sci}, 125: 545--558, 2015.

\bibitem{Gu}
Zuji Guo.
\newblock Elliptic equations with indefinite concave nonlinearities near the origin.
\newblock {\em Journal of Mathematical Analysis and
	Applications}, 367: 273--277, 2010.
    
\bibitem{Ha}
	Abdallah El Hamidi.
	\newblock Multiple solutions with changing sign energy to a nonlinear elliptic equation.
	\newblock {\em Communications on Pure and Applied Analysis}, 3(2): 253--265, 2004.
    
\bibitem{Il}
Yavdat Il'yasov.
\newblock On nonlocal existence results for elliptic equations with convex–concave nonlinearities.
\newblock {\em  Nonlinear Analysis: Theory, Methods \& Applications}, 61(1-2):211--236, 2005.









%
%


\bibitem{KaQuUm}
Uriel Kaufmann, Humberto Ramos~Quoirin and Kenichiro Umezu.
\newblock Loop Type Subcontinua of Positive Solutions for Indefinite Concave-Convex Problems.
\newblock {\em Advanced Nonlinear Studies}, 19(2): 391--412, 2019.


\bibitem{LeQuSi}
Edir Júnior Ferreira Leite, Humberto Ramos~Quoirin and Kaye Silva.
\newblock Further applications of the Nehari manifold method to functionals in $C^1(X\setminus\{0\})$.
\newblock {\em Nonlinear Analysis}, 261: 113890, 2025.


%










%
%
%


%
%

%
\bibitem{PeAgOr}
Kanishka Perera, Ravi~P. Agarwal, and Donal O'Regan.
\newblock {\em Morse theoretic aspects of {$p$}-{L}aplacian type operators},
volume 161 of {\em Mathematical Surveys and Monographs}.
\newblock American Mathematical Society, Providence, RI, 2010.









\bibitem{Pa}
Francisco Odair de Paiva.
\newblock Nonnegative solutions of elliptic problems with sublinear indefinite nonlinearity.
\newblock {\em Journal of Functional Analysis}, 261: 2569--2586, 2011.

\bibitem{MR4736027}
Humberto Ramos~Quoirin, Gaetano Siciliano, and Kaye Silva.
\newblock Critical points with prescribed energy for a class of functionals
depending on a parameter: existence, multiplicity and bifurcation results.
\newblock {\em Nonlinearity}, 37(6):Paper No. 065010, 40, 2024.

%

\bibitem{SiMa}
Kaye Silva and Abiel Macedo.
\newblock Local minimizers over the Nehari manifold for a class of concave-convex problems with sign changing nonlinearity.
\newblock {\em J. DifferentialEquations}, 265: 1894--1921, 2018.

\bibitem{Sz}
Andrzej Szulkin.
\newblock Ljusternik--Schnirelmann theory on $C^1$-manifolds.
\newblock {\em Ann. Inst. Henri Poincare}, 5(2):119--139, 1988.


\bibitem{Wu}
Tsung-fang Wu.
\newblock Multiple positive solutions for a class
of concave–convex elliptic problems in $\mathbb{R}^N$ involving
sign-changing weight.
\newblock {\em Journal of Functional Analysis}, 258:99--131, 2010.
	
\end{thebibliography}
\end{document}